\documentclass[10pt,reqno]{amsart}

\usepackage{amsmath}
\usepackage{amsthm}
\usepackage{amssymb}
\usepackage{amsfonts,mathrsfs}
\usepackage{stmaryrd}
\usepackage{amsxtra}  
\usepackage{epsfig}
\usepackage{verbatim}
\usepackage{xcolor}
\usepackage{graphicx}
\usepackage[utf8]{inputenc} % usually not needed (loaded by default)
\usepackage[T1]{fontenc}
\usepackage[makeroom]{cancel}
\usepackage{csquotes}
\usepackage{enumitem}
\usepackage{hyperref}
 \usepackage[normalem]{ulem}
\usepackage{cleveref}

\theoremstyle{plain}
\newtheorem{theorem}[equation]{Theorem}
\newtheorem{proposition}[equation]{Proposition}
\newtheorem{lemma}[equation]{Lemma}
\newtheorem{corollary}[equation]{Corollary}
\newtheorem{definition}[equation]{Definition}
\theoremstyle{definition}

\theoremstyle{remark}
\newtheorem{remark}[equation]{Remark}
\newtheorem{example}[equation]{Example}
\numberwithin{equation}{section}

\newcommand{\sjump}{\hskip .2 cm}

\newcommand{\im}{\operatorname{Im}}
\newcommand{\re}{\operatorname{Re}}

\newcommand{\Hcomplex}{H^{\scriptscriptstyle\mathbb{C}}}
\newcommand{\Hc}[3]{H^{\scriptscriptstyle\mathbb{C}}_{#1}({#2},{#3})}
\newcommand{\Hr}[3]{Q^{\scriptscriptstyle\mathbb{R}}_{#1}({#2},{#3})}
\newcommand{\Qrsymbol}[1]{Q^{\scriptscriptstyle\mathbb{R}}_{#1}}
\newcommand{\Qrsymbolcal}[1]{\mathcal{Q}^{\scriptscriptstyle\mathbb{R}}_{#1}}
\newcommand{\Qc}[3]{Q^{\scriptscriptstyle\mathbb{C}}_{#1}({#2},{#3})}
\newcommand{\Hcsymbol}[1]{H^{\scriptscriptstyle\mathbb{C}}_{#1}}
\newcommand{\Hcsymbolcal}[1]{\mathcal{H}^{\scriptscriptstyle\mathbb{C}}_{#1}}
\newcommand{\Tr}[2]{T_{#1}({#2})}
\newcommand{\Tc}[2]{T_{#1}({#2})^{\scriptscriptstyle\mathbb{C}}}
\newcommand{\Thol}[2]{T_{#1}({#2})^{1,0}}
\newcommand{\Tahol}[2]{T_{#1}({#2})^{0,1}}
\newcommand{\Vr}[2]{\mathcal{V}_{#1}({#2})}
\newcommand{\Vc}[2]{\mathcal{V}_{#1}({#2})^{\scriptscriptstyle\mathbb{C}}}
\newcommand{\Vhol}[2]{\mathcal{V}_{#1}({#2})^{1,0}}
\newcommand{\Vahol}[2]{\mathcal{V}_{#1}({#2})^{0,1}}
\newcommand{\Tastr}[2]{T^\ast_{#1}({#2})}
\newcommand{\Tastc}[2]{T^\ast_{#1}({#2})_{\scriptscriptstyle\mathbb{C}}}
\newcommand{\Tasthol}[2]{T^\ast_{#1}({#2})_{1,0}}
\newcommand{\Tastahol}[2]{T^\ast_{#1}({#2})_{0,1}}
\newcommand{\Omegar}[2]{\Omega^1_{#1}({#2})}
\newcommand{\Omegac}[2]{\Omega^1_{#1}({#2})_{\scriptscriptstyle\mathbb{C}}}
\newcommand{\Omegahol}[2]{\Omega^1_{#1}({#2})_{1,0}}
\newcommand{\Omegaahol}[2]{\Omega^1_{#1}({#2})_{0,1}}
\newcommand{\coloneqq}{\mathrel{\mathop:}=}
\newcommand{\eqqcolon}{=\mathrel{\mathop:}}

\DeclareMathOperator{\RE}{Re}
\DeclareMathOperator{\IM}{Im}

\DeclareMathOperator{\GRAD}{grad}
\DeclareMathOperator{\TRACE}{tr}

\providecommand{\abs}[1]{\lvert#1\rvert}

\providecommand{\ang}[1]{\langle#1\rangle}

\providecommand{\zbar}{{\bar{z}}}
\providecommand{\wbar}{{\bar{w}}}

\newcommand{\N}{\mathbb{N}}

\newcommand{\R}{\mathbb{R}}
\newcommand{\C}{\mathbb{C}}
\newcommand{\D}{\mathbb{D}}

\begin{document}

\title[Plurisubharmonic defining functions]{On plurisubharmonic defining functions for pseudoconvex domains in $\mathbb{C}^2$}
\author{Anne-Katrin Gallagher}
\address{Gallagher Tool and Instrument LLC, Redmond, WA 98052, USA}
\email{anne.g@gallagherti.com}
\author{Tobias Harz}
\address{Bergische Universitat Wuppertal, Wuppertal, Germany}
\email{harz@math.uni-wuppertal.de}
\date{\today}
%\thanks{}

\begin{abstract}
We investigate the question of existence of  plurisubharmonic defining functions for smoothly bounded, 
pseudoconvex domains in $\mathbb{C}^2$. In particular, we construct a family of simple counterexamples  
to the existence of plurisubharmonic smooth local defining functions. Moreover, we give general criteria 
equivalent to the existence of plurisubharmonic smooth defining functions on or near the boundary of the domain. 
These equivalent characterizations are then explored for some classes of domains.
\end{abstract}
\maketitle

\section{Introduction}

In this paper, we investigate the question of existence of plurisubharmonic smooth defining functions for pseudoconvex domains with smooth boundary. This basic question has been long resolved for strictly pseudoconvex domains. However, there is a great lack of understanding in the case of weakly pseudoconvex domains, although the existence of plurisubharmonic defining functions  is of relevance, e.g.,  for the classification problem of domains in $\mathbb{C}^n$.
%The question of existence of plurisubharmonic defining functions for pseudoconvex domains is a basic problem in Complex Analysis. However, there is a great lack of understanding of this question for weakly pseudoconvex domains, although the existence of plurisubharmonic defining functions  is of relevance, e.g., for the classification problem of domains in $\mathbb{C}^n$.

It is a basic fact that a smoothly bounded domain in $\C^n$ has a plurisubharmonic smooth local defining function 
near a boundary point if the domain is strictly pseudoconvex or convex near that point; similarly, the domain 
admits a smooth plurisubharmonic defining function near the boundary if it is strictly pseudoconvex or convex 
at each boundary point. 
No other geometric conditions which are sufficient for the existence of plurisubharmonic smooth (local) defining functions
are known.
Moreover, the existence of local or global plurisubharmonic defining functions may fail on pseudoconvex domains with
weakly pseudoconvex boundary points.
For instance, the worm domain, constructed by Diederich--Forn{\ae}ss in \cite{DieFor77-1}, does not admit 
a plurisubharmonic global defining function, although it does admit plurisubharmonic local defining functions near each
boundary point. 
Further,  Forn{\ae}ss \cite{Fornaess79} constructs a smoothly bounded, pseudoconvex domain in $\mathbb{C}^3$ for which all $\mathcal{C}^2$-smooth local defining functions fail to 
be plurisubharmonic on the boundary near some 
boundary point. Behrens \cite{Behrens84} gives an example of a pseudoconvex domain in $\mathbb{C}^2$ with real-analytic boundary which exhibits the same failure of plurisubharmonicity of $\mathcal{C}^6$-smooth local defining 
functions. Behrens' domain is of type $6$ at the boundary point in question. 
No such examples are known for pseudoconvex domains which are of type $4$ at the considered boundary point.

In the first part of this paper, we construct a family of counterexamples in the spirit of Behrens' example in \cite{Behrens84}. Namely, we construct domains $\Omega_{2k}\subset\mathbb{C}^2$, $k\geq3$, such that  $\Omega_{2k}$
is a domain with real-analytic boundary that is pseudoconvex and of type $2k$ at some boundary point $p$ but any $\mathcal{C}^2$-smooth defining function of $\Omega$ near $p$ fails to be plurisubharmonic on the boundary. 

The second part of the paper is concerned with introducing new geometric conditions that are sufficient for the existence of plurisubharmonic smooth (local) defining functions. We first give equivalent, yet non-geometric, characterizations for the existence of plurisubharmonic smooth local defining functions, both on and near the boundary, see \Cref{L:pshonboundaryiff} and \Cref{L:pshnearboundaryiff}. These characterizations are then exploited to show the existence of plurisubharmonic defining functions under each of two newly introduced geometric conditions for smoothly bounded pseudoconvex domains.

The first condition pertains to type $4$ boundary points $p$ of pseudconvex domains $\Omega\subset\mathbb{C}^2$ with smooth boundary. We first show that the Levi form $\lambda$ of $\Omega$ near $p$ satisfies
\begin{align}\label{E:Classinequality}
  (L\bar{L}\lambda)(p)\geq |(LL\lambda)(p)|
\end{align}
at points $p\in b\Omega$ where $\Omega$ is weakly pseudoconvex, for all holomorphic tangential vector fields $L$ near $p$, see Theorem \ref{T:basictype4estimate}. 
Note that \eqref{E:Classinequality} holds trivially at all boundary points at which $\Omega$ is of type larger than $4$. Through \eqref{E:Classinequality}, we may then classify type $4$ boundary points of $\Omega$ into two groups as follows. We say that $\Omega$ is of \emph{strict type $4$} at $p$ if
inequality \eqref{E:Classinequality} is strict, otherwise, call $\Omega$ of \emph{weak type $4$} at $p$. This classification is easily seen to be invariant under biholomorphic coordinate changes, see Lemma \ref{thm:strictestimatetransform}.
We then give a description of pseudoconvex domains near strict type $4$ boundary points in suitable local holomorphic coordinates in Proposition \ref{thm:stricttype4incoordinates}. In this form, the notion of strict type $4$ boundary points has appeared, albeit implicitly, in the literature. Namely, Kol\'a\v{r} showed in \cite{Kolar99}
that a pseudoconvex domain with real-analytic boundary is convexifiable near any strict type $4$ boundary point. He also states in \cite{Kolar99} that this result does not hold if the real-analyticity assumption is replaced by mere $\mathcal{C}^\infty$-smoothness of the boundary. A proof of this statement is not provided in \cite{Kolar99}.

We show that if a smoothly bounded, pseudoconvex domain $\Omega\subset\mathbb{C}^2$ is of strict type $4$ at some boundary point $p$, then there exists a  smooth local defining function for $\Omega$ near $p$ which is plurisubharmonic on $b\Omega$, see Theorem \ref{T:type4pshonboundary}. This local defining function is constructed explicitly by solving the following first order partial differential equation on $b\Omega$ near $p$. For a given, smooth, complex-valued function $F$ near $p$, find a smooth, real-valued function $h$ near $p$ such that
\begin{align*}
  Lh=F +\mathcal{O}(\sqrt{\lambda})\;\;\text{ on }b\Omega\text{  near }p.
\end{align*}

We also prove that if $\Omega\subset\mathbb{C}^2$ admits a smooth defining function which is plurisubharmonic on $b\Omega$ near a type $4$ boundary point $p$, then $\Omega$ admits a plurisubharmonic smooth defining function, see Theorem \ref{T:pshnearboundary}. To construct this defining function, we solve the following system of partial differential equations on $b\Omega$ near $p$. For a given, smooth, real-valued function $F$ near $p$, find a smooth, real-valued function $h$ near $p$ such that
\[
\begin{cases}
    &Lh=\mathcal{O}(\sqrt{\lambda})\\
    &L\bar{L}h=F+\mathcal{O}(\sqrt{\lambda})
\end{cases}
\;\;\text{ on }b\Omega\text{  near }p.
\]
This method of proof fails for higher order boundary points. However, parts of our analysis can be salvaged to give a simplified, short proof of the fact that the Diederich--Forn{\ae}ss index and Steinness index of $\Omega\subset\mathbb{C}^2$ are $1$ if $\Omega$ admits a smooth defining function that is plurisubharmonic on $b\Omega$, but does not have type $4$ boundary points, see Corollary \ref{C:DFforhighertype}.

For the second application of our general characterizations of the existence of plurisubharmonic defining functions, see Propositions \ref{L:pshonboundaryiff} and \ref{L:pshnearboundaryiff}, we introduce the notion of sesquiconvexity, a new geometric condition for an open set in $\C^2$ which may be formulated independently of the choice of a defining function. Sesquiconvexity is sufficient for the existence of (local) defining functions that are plurisubharmonic on the boundary
(and inside the domain), 
see \Cref{P:sesquiconvexpsh} and \Cref{T:sesquiconvexpshglobal}, however, it is not a necessary condition, see \Cref{E:examplecont}.
%Lastly, we introduce the notion of sesquiconvexity, a new geometric condition for an open set which may be formulated
%independently of the choice of the defining function. Sesquiconvexity is sufficient for the existence of plurisubharmonic
%(local) defining functions, however, it is not a necessary condition, see \Cref{E:examplecont}.

%
%
%
\section{Preliminaries}\label{sec:preliminaries}

In this section, we detail our notations and list some known facts for later reference. 
We note that \Cref{sec:counterexample} requires almost no prerequisites. As such, readers familiar with basic notions from several complex variables may delay looking through the current section until the study of \Cref{sec:stricttype4} and onward.

%For the majority of the paper, we will work in the $\mathcal{C}^\infty$-smooth category. 
The generic term \enquote{smooth} always means $\mathcal{C}^\infty$-smooth. Domains and functions of finite smoothness classes will only be considered in \Cref{sec:counterexample}.
\subsection{Some basic notions from almost complex geometry}
Let $(M, J)$ be an almost complex manifold. For every $p \in M$, write $\Tr{p}{M}$ for the real tangent space to $M$ at $p$. As usual, the complexification $\Tc{p}{M} \coloneqq T_p(M) \otimes \C$ decomposes as $\Tc{p}{M} = \Thol{p}{M} \oplus \Tahol{p}{M}$ into the $+i$ and $-i$ eigenspaces of $J_p$.
%(more precisely, its $\C$-linear extension to $\Tc{p}{M}$)
Smooth sections over $M$ in the corresponding bundles $\Tr{}{M}$, $\Tc{}{M}$, $\Thol{}{M}$, $\Tahol{}{M}$ are denoted by $\Vr{}{M}$, $\Vc{}{M}$, $\Vhol{}{M}$, $\Vahol{}{M}$, and are called, real, complexified, holomorphic and antiholomorphic vector fields on $M$, respectively. 

Let $\Tastr{p}{M}$ denote the real cotangent space of $M$ at $p$, and let $J^\ast_p \colon \Tastr{p}{M} \to \Tastr{p}{M}$ be the dual almost complex structure. %$J_p^\ast(\alpha_p)(V_p) \coloneqq \alpha_p(J_pV_p)$, $V_p \in \Tr{p}{M}$, $\alpha_p \in \Tastr{p}{M}$.
% i.e., $\ang{J^\ast_p\alpha_p,V_p} \coloneqq \ang{\alpha_p, J_pV_p}$ for $\alpha_p \in T^\ast_\R(M)$, and $V_p \in \Vr{p}{M}$; here $\ang{\cdot\,\,\,\cdot}$ denotes the natural pairing of 1-forms and vector fields.
The complexification $\Tastc{p}{M} \coloneqq \Tastr{p}{M} \otimes \C$ decomposes as $\Tastc{p}{M} = \Tasthol{p}{M} \oplus \Tastahol{p}{M}$ into the $+i$ and $-i$ eigenspaces of $J^\ast_p$.
%(more precisely, its $\C$-linear extension to $\Tastc{p}{M}$)
Smooth sections over $M$ in the corresponding bundles $\Tastr{}{M}$, $\Tastc{}{M}$, $\Tasthol{}{M}$, $\Tastahol{}{M}$ are denoted by $\Omegar{}{M}$, $\Omegac{}{M}$, $\Omegahol{}{M}$, $\Omegaahol{}{M}$, and are called, real, complexified, holomorphic and antiholomorphic 1-forms on $M$, respectively.

For $\alpha \in \Omegar{}{M}$, write $\alpha_{\scriptscriptstyle\C}$ to denote the pointwise $\C$-linear extension of $\alpha$ to $\Omegac{}{M}$. Then $\alpha_{1,0} \coloneqq \alpha_{\scriptscriptstyle\C}-i(J^\ast\alpha)_{\scriptscriptstyle\C}$ and $\alpha_{0,1} \coloneqq \alpha_{\scriptscriptstyle\C}+i(J^\ast\alpha)_{\scriptscriptstyle\C}$ define elements in $\Omegahol{}{M}$ and $\Omegaahol{}{M}$, respectively.
If $f \colon M \to \C$ is a smooth function, write $df \in \Omegar{}{M}$ to denote the differential of $f$, and set $\partial f \coloneqq \textstyle\frac{1}{2}(df)_{1,0}$ and $\bar\partial f \coloneqq \textstyle\frac{1}{2}(df)_{0,1}$.
Clearly, $(df)_{\scriptscriptstyle\C} = \partial f + \bar\partial f$ holds.

If $M$ is an open subset in $\C^n$ with coordinates $(z^1, \ldots, z^n)$, $z^j =x^j+iy^j$, $x^j,y^j \in \R$, and if $J$ is the standard almost complex structure, then
\begin{align*}
\partial f &= \sum_{j=1}^n \frac{\partial f}{\partial z^j} dz^j, &
\bar\partial f &= \sum_{j=1}^n \frac{\partial f}{\partial \bar{z}^j} d\bar{z}^j,
\end{align*}
where 
\begin{align*}
\frac{\partial}{\partial z^j} &= \frac{1}{2}\left(\frac{\partial}{\partial x^j} - i\frac{\partial}{\partial y^j}\right),  &
\frac{\partial}{\partial \bar{z}^j} &= \frac{1}{2}\left(\frac{\partial}{\partial x^j} + i\frac{\partial}{\partial y^j}\right),\\
\intertext{and}
dz^j &= dx^j + idy^j, & d\bar{z}^j &= dx^j - idy^j.
\end{align*}
Here, and in what follows, we always suppress the index $\C$ and use the same symbol to denote the vector field $V \in \Vr{}{M}$ and its complexification $V^{\scriptscriptstyle \C} \coloneqq V \otimes 1 \in \Vc{}{M}$, and the 1-form $\alpha \in \Omegar{}{M}$ and its complexification $\alpha_{\scriptscriptstyle \C} \in \Omegac{ }{M}$, respectively.
Euclidean inner products on $\Vr{}{M}$ and $\Omegar{}{M}$ are introduced by declaring that $(\frac{\partial}{\partial x^1}, \frac{\partial}{\partial y^1}, \ldots, \frac{\partial}{\partial x^n}, \frac{\partial}{\partial y^n})$ and $(dx^1, dy^1, \ldots, dx^n, dy^n)$ are orthonormal bases for $\Vr{}{M}$ and $\Omegar{}{M}$, respectively. Similarly, Hermitian inner products on $\Vc{}{M}$ and $\Omegac{}{M}$ are introduced by declaring that $(\frac{\partial}{\partial z^1}, \ldots, \frac{\partial}{\partial z^n}, \frac{\partial}{\partial \bar{z}^1}, \ldots, \frac{\partial}{\partial \bar{z}^n})$ is an orthogonal basis for $\Vc{}{M}$ with vectors of constant length $1/\sqrt{2}$, and $(dz^1, \ldots, dz^n,$ $d\bar{z}^1, \ldots, d\bar{z}^n)$ is an orthogonal basis for $\Omegac{}{M}$ with vectors of constant length $\sqrt{2}$. 
Note that the canonical inclusions $\Vr{}{M} \hookrightarrow \Vc{}{M}$, $V \mapsto V^{\scriptscriptstyle \C}$, and $\Omegar{}{M} \hookrightarrow \Omegac{}{M}$, $\alpha \mapsto \alpha_{\scriptscriptstyle \C}$, are isometries, so the notations $\abs{V}$ and $\abs{\alpha}$ for the corresponding norms are well-defined, even if the index $\C$ is suppressed in the notation. % even if we write $V = V^{\scriptscriptstyle \C}$ and $\alpha = \alpha_{\scriptscriptstyle \C}$.

\subsection{Some basic notions from differential geometry}
Let $M$ be a smooth manifold. For every tensor field $\gamma$ on $M$ and every $p \in M$, let $\gamma_p = \gamma(p)$ denote the value of $\gamma$ at $p$. If $\alpha \in \Omegar{}{M}$ and $V \in \Vr{}{M}$, then $\ang{\alpha,V} \coloneqq \alpha(V)$.
%We write $\ang{\,\cdot\,,\,\cdot\,}$ for the natural pairing of 1-forms and vector fields. 
Given smooth vector fields $V_1, \ldots, V_k$ on $M$ and a smooth function $f \colon M \to \C$, we write $V_k \ldots V_1f \coloneqq V_k( \ldots (V_1f)\ldots)$. Moreover, $[V_1,V_2]$ denotes the Lie bracket of $V_1$ and $V_2$, i.e., $[V_1,V_2]f = V_1V_2f - V_2V_1f$.

Let $M \subset \R^N$ be open and let $(t^1, \ldots, t^N)$ be the standard Euclidean coordinates on $M$. For $V,W \in \Vr{}{M}$, $V = \sum_{\nu=1}^N V^\nu\frac{\partial}{\partial t^\nu}$, $W = \sum_{\nu=1}^N W^\nu\frac{\partial}{\partial t^\nu}$, we set
$$
\nabla_VW = \sum_{\nu=1}^N \bigg(\sum_{\mu=1}^N V^\mu\frac{\partial W^\nu}{\partial t^\mu}\bigg) \frac{\partial}{\partial t^\nu}.
$$
Note that $\nabla_VW$ is precisely the covariant derivative of $W$ along $V$ with respect to the Levi-Civita connection in $\Tr{}{M}$ corresponding to the standard Riemannian metric $g = \sum_{\nu = 1}^N dt^\nu \otimes dt^\nu$ on $\Tr{}{M}$. If for $f \in \mathcal{C}^\infty(M,\R)$ we write
$$
\Hr{f}{V}{W} = \sum_{\nu,\mu = 1}^N \frac{\partial^2 f}{\partial t^\nu \partial t^\mu} V^\nu W^\mu,
$$
then
\begin{align} \label{equ:XYf}
VWf = \Hr{f}{V}{W} + (\nabla_VW)f.
\end{align}
Let $I \subset \R$ be an open interval, and let $\gamma \colon I \to M$ be a smooth curve. Then $\dot{\gamma}, \ddot{\gamma} \colon I \to \Tr{}{M}$ denote the vector fields along $\gamma$ given by $\dot{\gamma}(\tau) \coloneqq \sum_{\nu=1}^N \gamma_\nu'(\tau)(\frac{\partial}{\partial t_j})_{\gamma(\tau)}$ and  $\ddot{\gamma}(t) \coloneqq \sum_{\nu=1}^N \gamma_\nu''(\tau)(\frac{\partial}{\partial t_j})_{\gamma(\tau)}$. Observe that $\ddot{\gamma}$ is the covariant derivative of $\dot{\gamma}$ with respect to $\nabla$.

Let $M \subset \C^n$ be open and let $(z^1, \ldots, z^n)$ be the standard Euclidean coordinates on $M$. For $V,W \in \Vhol{}{M}$, $V = \sum_{j=1}^n V^j\frac{\partial}{\partial z^j}$, $W = \sum_{j=1}^n W^j\frac{\partial}{\partial z^j}$, we set
\begin{equation*}\begin{split}
\nabla_VW = \sum_{j=1}^n \bigg(\sum_{k=1}^n V^k\frac{\partial W^j}{\partial z^k}\bigg) \frac{\partial}{\partial z^j}, \quad \nabla_{\bar{V}}\bar{W} \coloneqq \overline{\nabla_VW},  \\
\nabla_{\bar{V}}W = \sum_{j=1}^n \bigg(\sum_{k=1}^n \bar{V}^k\frac{\partial W^j}{\partial \zbar^k}\bigg) \frac{\partial}{\partial z^j},\quad \nabla_V\bar{W} \coloneqq \overline{\nabla_{\bar{V}}W}.
\end{split}\end{equation*}
Note that $\nabla_VW$ and $\nabla_{\bar{V}}W$ are precisely the covariant derivatives of $W$ along $V$ and $\bar{V}$ with respect to the Chern connection in $\Thol{}{M}$ corresponding to the standard Hermitian metric $h = \sum_{j = 1}^n dz^j \otimes d\zbar^j$ on $\Thol{}{M}$, respectively. If for $f \in \mathcal{C}^\infty(M,\C)$ we write
\begin{equation*}\begin{split}
\Qc{f}{V}{W} &= \sum_{j,k = 1}^n \frac{\partial^2 f}{\partial z^j \partial z^k} V^j W^k, \\
\Hc{f}{V}{W} &= \sum_{j,k = 1}^n \frac{\partial^2 f}{\partial z^j \partial \zbar^k} V^j \bar{W}^k,
\end{split}\end{equation*}
then
\begin{align}
\label{equ:LNf}
VWf &= \Qc{f}{V}{W} + (\nabla_VW)f, \\
\label{equ:LNbarf}
V\bar{W}f &= \Hc{f}{V}{W} + (\nabla_V\bar{W})f.
\end{align}
Observe that the above formula for $\Hc{f}{V}{W}$ defines a map 
$$
\Hcsymbol{f} \colon \Vhol{}{M} \times \Vhol{}{M} \to \C, \quad %\text{ such that }
\Hcsymbol{f}(V,W) = \partial\bar{\partial}f(V,\bar{W}).
$$
If $f$ is real-valued, then $\Hcsymbol{f}$ is a sesquilinear form on the $\mathcal{C}^\infty(M,\C)$-module $\Vhol{}{M}$. %The notation $\Hcsymbol{f}(L,N)$ is consistent with formula (\ref{equ:defHfc}) above.

The above notations for the Levi-Civita connection and the Chern connection on an open set $M \subset \R^{2n} \simeq \C^n$ are unambiguous in the following sense. If $A,B,C,D \in \Vr{}{M}$ such that $C+iD \in \Vhol{}{M}$, then
$$
\nabla_{A+iB}(C+iD) = (\nabla_AC + i \nabla_AD) + i(\nabla_BC + i\nabla_BD),
$$
where on the left-hand side $\nabla$ denotes the Chern connection, and on the right-hand side $\nabla$ denotes the Levi-Civita connection.

\subsection{Defining functions, pseudoconvexity, and finite type}
Let $\Omega \subset \C^n$ be a $\mathcal{C}^2$-smoothly bounded domain and let $p_0$ in $b\Omega$. We say that a $\mathcal{C}^2$-smooth function $r \colon U \to \R$ is a defining function for $\Omega$, if
\begin{enumerate}[label=(\roman*)]
  \item $U$ is an open neighborhood of $b\Omega$,
  \item $\Omega \cap U = \{r < 0\}$,
  \item $dr \neq 0$ on $b\Omega$.
\end{enumerate}
Moreover, we say that a smooth function $r \colon U \to \R$ is a local defining function for $\Omega$ (near $p_0$), if $U$ is an open neighborhood of $p_0$, and $r$ satisfies (ii) and (iii).
%if it is the restriction of a defining function $\hat{r} \colon \hat{U} \to \R$ for $\Omega$ to an open neighborhood $U \Subset \hat{U}$ of $p_0$, see \Cref{R:Landau} for the main reason of adopting this convention.
%Moreover, we say that a smooth function $r \colon U \to \R$ is a local defining function for $\Omega$ (near $p_0$) if 
%\begin{enumerate}
%  \item $U$ is an open neighborhood of $p_0$, and $r$ extends to a smooth function on an open set $\hat{U}$ such that $U \subset\subset \hat{U}$,
%  \item $\Omega \cap U = \{r < 0\}$,
%  \item $dr \neq 0$ on $U$.
%\end{enumerate}
%Note that if $\Omega$ is bounded and $r \colon U \to \R$ is a defining function for $\Omega$, then, after suitably shrinking $U$, $r$ is also a local defining function for $\Omega$ (near every point $p_0 \in b\Omega$).

%Moreover, we say that a smooth function $r \colon U \to \R$ is a local defining function for $\Omega$ (near $p_0$) if 
%\begin{enumerate}[label=(\roman*)]
%  \item $U$ is an open neighborhood of $p_0$,
%  \item $\Omega \cap U = \{r < 0\}$,
%  \item $dr \neq 0$ on $b\Omega \cap U$.
%\end{enumerate}
%Clearly, if $r \colon U \to \R$ is a defining function for $\Omega$, then it is also a local defining function for $\Omega$ near every point $p_0 \in b\Omega$. 

For every $p \in b\Omega$, set $\Thol{p}{b\Omega} \coloneqq \Thol{p}{\C^n}  \cap \Tc{p}{b\Omega}$. A vector field $L \in \Vhol{}{U}$ is called tangential if $L_{p} \in \Thol{p}{b\Omega}$ for every $p \in b\Omega \cap U$. If $r \colon U \to \R$ is a local defining function for $\Omega$, then $\Thol{p}{b\Omega} = \{L_p \in \Thol{p}{\C^n} : \langle \partial r,L \rangle(p) = 0\}$ for every $p \in b\Omega \cap U$, and $L$ is tangential if and only if $Lr \equiv 0$ on $b\Omega \cap U$. 

The domain $\Omega$ is called pseudoconvex at $p$ if $\Hc{r}{L}{L}_{|_{b\Omega \cap U}} \ge 0$ near $p$ for all tangential vector fields $L \in \Vhol{}{U}$ near $p$. % L \in \mathcal{V}^{1,0}(U).
It is called strictly pseudconvex at $p$ if in the previous condition the inequality is strict whenever $L$ is nonvanishing.
%\Hc{r}{L}{L}(p) &> 0 \text{ for all tangential } L \in \mathcal{V}^{1,0}(U), L \neq 0.
%\end{align*}
%Note that $\Hcsymbol{r}$ is a tensor, thus in order to prove (strict) pseudoconvexity of $\Omega$ at $p$ it suffices to check definiteness of $(\Hcsymbol{r})_p$ on $T^{1,0}_p(b\Omega)$.
If $\Omega$ is pseudoconvex at $p$ but not strictly pseudoconvex at $p$, then $\Omega$ is said to be weakly pseudoconvex at $p$.

Let $\Omega \subset \C^2$ be a smoothly bounded domain, and let $p \in b\Omega$. Let $r \colon U \to \R$ be a smooth local defining function for $\Omega$ near $p$, let $L$ be a nonvanishing holomorphic tangential vector field near $p$, and let $\lambda \coloneqq \Hc{r}{L}{L}_{|_{b\Omega \cap U}}$. Then the following conditions are equivalent, and independent of the choices of $r$ and $L$.
\begin{enumerate}
 \item There exists $k \in \N$ such that 
 $$
 \exists\, L_1, \ldots, L_k \in \{L,\bar{L}\} : \left\langle \partial r, [\ldots[[L_1,L_2], L_3], \ldots, L_k] \right\rangle(p) \neq 0,
 $$
 and $k$ is the smallest integer with this property. 
 \item There exists $k \in \N$ such that 
 $$
 \exists\, L_1, \ldots, L_{k-2} \in \{L,\bar{L}\} : (L_{k-2} \ldots L_1 \lambda)(p) \neq 0,
 $$
 and $k$ is the smallest integer with this property.
 \item There exists $k \in \N$ and local holomorphic coordinates $z,w$ centered at $p$, such that
 \begin{align} \label{E:FiniteTypeCoordinates}
 r(z,w) = \RE(w) + h(z,\bar{z}) + o(\abs{z}^{k}, \IM(w)),
 \end{align}
 where $h$ is a nonvanishing homogeneous polynomial of degree $k$ without pure terms. 
\end{enumerate}
The original definition (1) is given in \cite[Definition 2.3]{Kohn72}. For the equivalence of (1) and (2), see \cite[Proposition 2.8]{Kohn72}. The third characterization is implicitly contained in the proof of \cite[Lemma 3.16]{Kohn72}; see also \cite[Theorem 3.3]{Bloom78}.

If the above properties are satisfied, then $b\Omega$ is said to be of finite type at $p$, and the number $c_p \coloneqq c_p(b\Omega) \coloneqq k$ is called the type of $b\Omega$ at $p$. If $\Omega$ is pseudoconvex at $p$, then $c_p$ is an even number, see \cite[Theorem 3.1]{Kohn72}, $\Omega$ is strictly pseudoconvex at $p$ if and only if $c_p = 2$, and $\Omega$ is weakly pseudoconvex at $p$ if and only if $c_p \ge 4$.

Let $\Omega \subset \C^n$ be a smoothly bounded domain, and let $U \subset \C^n$ be open. A $\mathcal{C}^2$-smooth function $r \colon U \to \R$ is called plurisubharmonic if $\Hc{r}{V}{V} \ge 0$ for every $V \in \Vhol{}{U}$, and it is called strictly plurisubharmonic if $\Hc{r}{V}{V} > 0$ for every $V \in \Vhol{}{U}$, $V \neq 0$. Moreover, we say that $r$ is plurisubharmonic on $b\Omega \cap U$ if $\Hc{r}{V}{V}_{|_{b\Omega \cap U}} \ge 0$ for every $V \in \Vhol{}{U}$, and we say that $r$ is strictly plurisubharmonic on $b\Omega \cap U$ if $\Hc{r}{V}{V}_{|_{b\Omega \cap U}} > 0$ for every $V \in \Vhol{}{U}$, $V \neq 0$. Note that if $p \in b\Omega \cap U$ and $r$ is plurisubharmonic on $b\Omega \cap U$, then in general it does not follow that $r$ is plurisubharmonic on any open neighborhood $U' \subset U$ of $p$, even if $r$ is a local defining function for $\Omega$.

%\noindent \framebox{To change: Add finite type, (strictly) psh, (strictly psh) on the boundary}

\subsection{Canonical vector fields in $\C^2$} \label{sec:CanonicalVectorFields}
Let $\Omega \subset \C^2$ be a smoothly bounded domain, and let $r \colon U \to \R$ be a smooth local defining function for $\Omega$. 
After possibly shrinking $U$, we may assume that $dr \neq 0$ on $U$. 
In this case, define vector fields $L_r,N_r \in \mathcal{V}^{1,0}(U)$ by
\begin{align}\label{E:definitionLN}
  L_r &= \frac{\sqrt{2}}{\abs{\partial r}}\left(r_w\frac{\partial}{\partial z}-r_z\frac{\partial}{\partial w}\right), \\
 \label{E:definitionN}
  N_r &= \frac{\sqrt{2}}{\abs{\partial r}}\left(r_{\bar{z}}\frac{\partial}{\partial z}+r_{\bar{w}}\frac{\partial}{\partial w}\right).
\end{align}
Then
\begin{enumerate}
  \item[(i)] $(L_r,N_r)$ is an orthogonal frame for $\Thol{}{U}$ such that $\abs{L_r} \equiv \frac{1}{\sqrt{2}} \equiv \abs{N_r} $,
  \item[(ii)] $L_rr \equiv 0$,% i.e., $L_r$ is complex tangential to the level sets of $r$,
  \item[(iii)] $N_rr = \frac{\abs{\partial r}}{\sqrt{2}}$.
\end{enumerate}
Note that, if $\rho$ is some other smooth local defining function for $\Omega$ on $U$, then
%If $\rho$ is some other smooth local defining function for $\Omega$ on $U$, then there exists
%a positive smooth function $h$ on $U$ such that $\rho = rh$. A straightforward computation shows that then 
%\begin{align}
%  \label{E:Lrho}
%  L_\rho &= \left(h \frac{\abs{\partial r}}{\abs{\partial \rho}} \right) L_r 
%  + r \cdot \left(\frac{\abs{\partial h}}{\abs{\partial \rho}} L_h\right) \quad \text{on } U,\\
%  \label{E:Nrho}
%  N_\rho &= \left(h \frac{\abs{\partial r}}{\abs{\partial \rho}} \right) N_r 
%  + r \cdot \left(\frac{\abs{\partial h}}{\abs{\partial \rho}} N_h\right) \quad \text{on } U.
%\end{align}
%Since $h = \abs{\partial \rho}/\abs{\partial r}$ on $b\Omega \cap U$, it follows, in particular, that 
\begin{align}\label{equ:independenceonboundary}
  L_\rho &= L_r  \quad\text{and}\quad N_\rho = N_r \quad \text{on } b\Omega \cap U.
\end{align}
%This is expected, as by properties (i)-(iii) above the values of the vector fields $L_r / S^1$ and $N_r$ at a point $q \in U$ are uniquely determined by the level surface of $r$ passing thorugh $q$. 
\noindent We will always use the notations $L = L_r$ and $N = N_r$ without explicit reference to the choice of the defining function $r$, if we consider these vector fields only on $b\Omega$. 

We will sometimes use the abbreviated notations $L$ and $N$ also for the vector fields on the whole open set $U$, if it is clear from the context which defining function we are referring to. As an example, given a fixed local defining function for $\Omega$ as above, we introduce $X, Y, T, \nu \in \Vr{}{U}$ to be the unique real vector fields such that
\begin{align*}
  L=\textstyle\frac{1}{2}\left(X+iY\right) \quad\text{and}\quad 
  N=\textstyle\frac{1}{2}\left(\nu+iT\right).
\end{align*}
It follows from properties (ii) and (iii) above, that 
\begin{align}\label{equ:canonicalframe}
  Xr = Yr = Tr =0 \quad\text{and}\quad \nu = \frac{\GRAD r}{\abs{\GRAD r}}.
\end{align}
By property (i), and since $Y = - JX$ and $T = -J\nu$%
%, where $J$ denotes the standard almost complex structure on $T^\C(U)$
, the vector fields $X, Y, T, \nu$ are linearly independent at each point $q \in U$. %here, $J$ denotes the standard almost complex structure on $U$.
In particular, it follows from (\ref{equ:independenceonboundary}) and (\ref{equ:canonicalframe}) that the restrictions of $X,Y,T$ to $b\Omega \cap U$ define a frame for $\Tr{}{b\Omega \cap U}$, which is independent of the choice of $r$.

 We write $\Hcsymbolcal{r}$ to denote the matrix associated with $\Hcsymbol{r} \colon \Vhol{}{U} \times \Vhol{}{U} \to \C$ relative to the basis $(L,N)$, i.e.,
\begin{align*}
 \Hcsymbolcal{r} \coloneqq
  \begin{pmatrix}
    \Hc{r}{L}{L} & \Hc{r}{L}{N} \\
    \Hc{r}{N}{L} & \Hc{r}{N}{N}
  \end{pmatrix}.
\end{align*}
The function $r$ is plurisubharmonic if and only if $\Hcsymbolcal{r}(q)$ is positive semi-definite at every point $q \in U$.

\subsection{Landau symbols}

Let $M$ be a smooth manifold, $p_0 \in M$, and $f,g \colon M \to \R$ be smooth functions. We use the usual notations
\begin{align*}
 f = o(g) \text{ for } p \to p_0 \quad&:\Leftrightarrow\quad \forall\; C > 0 \; : \abs{f} \le C \abs{g} \text{ in some neighborhood of }p_0,\\
   f = \mathcal{O}(g) \text{ for } p \to p_0 \quad&:\Leftrightarrow\quad \exists\; C > 0  : \abs{f} \le C \abs{g} \text{ in some neighborhood of }p_0.
\end{align*}
If it is clear from the context, we usually drop the explicit reference to the point $p_0$. Moreover, we write for $U\subset M$ open
\begin{align*}
 f = \mathcal{O}(g) \text{ on } U \quad&:\Leftrightarrow\quad \forall\; K \Subset U \,\exists\; C>0 : \abs{f} \le C\abs{g} \text{ on } K.
\end{align*}
Roughly speaking, the condition \enquote{\!$f=\mathcal{O}(g)$ on $U$} means that $f$ has at least the same order of vanishing on $U$ as $g$, but it does not contain any information about the growth of $f$ near the boundary of $U$. 
%In particular, note that if $f = \mathcal{O}(g)$ on $U$, then also $hf = \mathcal{O}(g)$ on $U$ for every smooth function $h \colon U \to \R$. 
Similarly, we write
\begin{align*}
 f \le \mathcal{O}(g) \text{ on } U \quad&:\Leftrightarrow\quad \forall\; K \Subset U \,\exists\; C>0 : f \le C\abs{g} \text{ on } K.
 \end{align*}
\begin{remark}\label{R:Landau}
  Let $\Omega \subset \C^2$ be a smoothly bounded pseudoconvex domain. If $r,\rho \colon U \to \R$ are two local defining functions for $\Omega$, and if $L,L'$ are two nonvanishing tangential holomorphic vector fields on $U$, then there exists a nonvanishing smooth function $h \colon U \to \R$ such that $\Hc{r}{L}{L}_{|_{b\Omega \cap U}} = h\Hc{\rho}{L'}{L'}_{|_{b\Omega \cap U}}$. In particular, if we write $\lambda \coloneqq \Hc{r}{L}{L}_{|_{b\Omega \cap U}}$, then the class of smooth functions $f \colon b\Omega \cap U \to \R$ such that, for example,
  $$
    f = \mathcal{O}(\lambda) \quad\text{on } b\Omega \cap U
  $$ 
  is well-defined, i.e., independent of the choice of $r$ and $L$.
  %a local defining function for $\Omega$ on $U$ and a tangential holomorphic vector field.
\end{remark}

%Let $M$ be a topological space. For every $p_0 \in M$, let $\mathcal{U}(p_0) \coloneqq \{U \subset M : U \text{ open neighborhood of } p_0\}$.
%For every $p_0 \in M$ and every constant $\varepsilon > 0$ we write $\mathbb{B}(p_0,\varepsilon) \coloneqq \{p \in M : d(p,p_0) < \varepsilon\}$. 
%If $f,g \colon M \to \R$ are given, then %we use the usual notations
%\begin{align*}
%f = o(g) \text{ (for } p \to p_0) \quad&:\Leftrightarrow\quad \forall\; C > 0 \;\exists\; U \in \mathcal{U}(p_0) : \abs{f} \le C \abs{g} \text{ on } U,\\
%   f = \mathcal{O}(g) \text{ (for } p \to p_0) \quad&:\Leftrightarrow\quad \exists\; C > 0 \;\exists\; U \in \mathcal{U}(p_0) : \abs{f} \le C \abs{g} \text{ on } U.
%\end{align*}
%Moreover, if $U \subset M$ is open, we write
%\begin{align*}
%f = \mathcal{O}(g) \text{ on } U \quad&:\Leftrightarrow\quad \forall\; K \Subset U \;\exists\; C>0 : \abs{f} \le C\abs{g} \text{ on } K.
%\end{align*}
%The following lemma is obvious.
%\begin{remark}\label{R:Landau}
%Let $f, g_1, g_2 \colon U \to \R$ be functions, and assume that $g_1 = hg_2$ for some nonvanishing continuous function $h \colon U \to \R$. Then one easily sees 
%that $f = \mathcal{O}(g_1)$ on $U$ if and only if $f = \mathcal{O}(g_2)$ on $U$. 
%\end{remark}

%Thus, roughly speaking, the condition $f = \mathcal{O}(g)$ on $U$ says that $f$ has the same order of vanishing as $g$ on $U$, but it does not contain any %information on the growth of $f$ at infinity.

\subsection{Miscellanea}
Let $\Omega \subset \R^N$ be a smoothly bounded domain. Let $d_{b\Omega} \colon \R^N \to [0,\infty)$,
$$ 
d_{b\Omega}(q) \coloneqq \inf_{p \in b\Omega} \abs{q-p},
$$
denote the Euclidean distance to the boundary $b\Omega$, and let $\delta_{b\Omega} \colon \R^N \to \R$,
$$
\delta_{b\Omega}(q) \coloneqq \begin{cases} d_{b\Omega}(q), & \text{ if } q \in \R^N \setminus \Omega \\ -d_{b\Omega}(q), & \text{ if } q \in \Omega \end{cases},
$$ 
be the associated signed distance function. Let $U \subset \R^N$ be an open neighborhood of $b\Omega$ such that there exists a smooth map $\pi \colon U \to b\Omega$ with $\abs{q-\pi(q)}=d_{b\Omega}(q)$,
see, e.g.,  \cite[Lemma 4.11]{Federer59} for existence and \cite[Lemma 1 in \S 15.5]{GilbargTrudinger} for smoothness of the map $\pi$. Finally, let $\nu$ denote the outward unit normal vector field along $b\Omega$, and note that this notation is consistent with the one given in \Cref{sec:CanonicalVectorFields}. 
If $f \colon U \to \R$ is smooth, then by Taylor's formula it follows that
\begin{align}\label{E:Taylor}
  f &= f \circ \pi +\delta_{b\Omega} ((\nu f) \circ \pi) +\mathcal{O}(d_{b\Omega}^{2}) \text{ on } U.
\end{align}
See also, e.g., (2.1) in \cite{ForHer07,ForHer08}.
%Then, by Taylor's formula, for every smooth $f \colon U \to \R$ and every $q \in U$, one has
%\begin{align}\label{E:Taylor}
%  f(q) &= f(\pi(q))+\delta_{b\Omega}(q) (\nu f)(\pi(q))+\mathcal{O}(d_{b\Omega}^{2})(q).
%\end{align}
%See also, e.g., (2.1) in \cite{ForHer07,ForHer08}.

If $f \colon \R \to \R$ is a nonnegative $\mathcal{C}^2$-smooth function and $f(x_0) = 0$, then it follows readily from L'Hospital's rule that $\abs{f'}^2 \le Cf$ near $x_0$ for some constant $C>0$. We will repeatedly need the following generalizations of this simple fact.

\begin{lemma} \label{thm:lHospital}
The following assertions hold true.
\begin{enumerate}
 \item Let $U \subset \R^N$ be open, and let $f \colon U \to [0,\infty)$ be a $\mathcal{C}^2$-smooth function. 
 Then for every $K \Subset U$ there exists a constant $C>0$ such that 
 \begin{align}\label{E:nonnegativefunctionestimate}
   |df|^2\leq Cf \text{ on } K.
 \end{align}

 \item Let $\Omega \subset \R^N$ be a smoothly bounded domain, let $U \subset \R^N$ be open, 
 and let $f \colon b\Omega \cap U \to [0,\infty)$ be a $\mathcal{C}^2$-smooth function. 
 %Then for every $K \Subset b\Omega \cap U$ there exists a constant $C>0$ such that for every 
 %$V \in \Vr{}{b\Omega \cap U}$, $\abs{V} \le 1$, one has
 %\begin{align}%\label{E:tangentialestimate}
 %   |Vf|^2 \leq Cf \text{ on } K.
 %\end{align}
 Then  
 \begin{align} \label{equ:VectorFieldlHospital}
   Vf = \mathcal{O}(\sqrt{f}) \text{ on } b\Omega \cap U
 \end{align} 
 for every $V \in \Vr{}{b\Omega \cap U}$.
\end{enumerate}
\end{lemma}
\begin{proof}
For a proof of part (1) see, e.g., \cite[Lemma 4.3]{ForHer08}. In order to prove part (2), let $F \colon U \to [0,\infty)$ be a $\mathcal{C}^2$-smooth extension of $f$, and for given $K \Subset b\Omega \cap U$ let $C>0$ be a constant such that $\abs{dF}^2 \le CF$ on $K$, see (\ref{E:nonnegativefunctionestimate}). Then $\abs{Vf}^2 = \abs{\ang{df,V}}^2 =  \abs{\ang{dF,V}}^2 \le \abs{dF}^2 \le CF = Cf$ on $K$ for every $V \in \Vr{}{b\Omega \cap U}$ such that $\abs{V} \le 1$, which implies (\ref{equ:VectorFieldlHospital}).
\end{proof}

%Let $a, b \in \C$ and let $\tau > 0$. Then the inequality $\abs{ab} \le \frac{1}{\tau}\abs{a}^2 + \tau\abs{b}^2$ will be referred to as the (sc)-(lc)-inequality.

%\noindent \framebox{To update: (1) Check in general. (2) Explain (sc)-(lc)-inequality?}

%
%
%
\section{A counterexample} \label{sec:counterexample}
Consider $\mathbb{C}^2$ with coordinates $(z,w)$, $z = x+iy$, $w = u+ iv$.
\begin{theorem}\label{T:counterexample}
  For fixed $k \in \N$, $k \ge 3$, let
  $$
  r(z,w)\coloneqq u +\textstyle \frac{1}{k^2}|z|^{2k} - \frac{2}{(k-1)^2} |z|^{2k-2}v + 
  \frac{1}{(k-2)^2}|z|^{2k-4}v^2 + |z|^{4k-2},
  $$
  and set 
  $$ 
  \Omega \coloneqq \big\{(z,w) \in \mathbb{C}^2 : r(z,w) < 0\big\}.
  $$
  Then the following assertions hold true.
  \begin{itemize}
  \item[(i)] There exists an open neighborhood $U \subset \C^2$ of $0 \in b\Omega$, such that $\Omega \cap U$ is pseudoconvex, and such that the following holds. If $k=3$, then $0 \in b\Omega$ is the only weakly pseudoconvex boundary point of $\Omega$ in $b\Omega \cap U$. If $k>3$, then the set of weakly pseudoconvex boundary points of $\Omega$ in $b\Omega \cap U$ is $(b\Omega \cap U) \cap (\{0\} \times \C)$. Moreover, $b\Omega$ is of finite type $c_0 = 2k$ at $0$.
  \item[(ii)] Let $V \subset \C^2$ be an open neighborhood of $0$ and let $\rho \colon V \to \R$ be a $\mathcal{C}^2$-smooth local defining function for $\Omega$. Then $\rho$ is not plurisubharmonic on $b\Omega \cap V$.
  \end{itemize}  
\end{theorem}

\begin{proof}
   (i) In a slight deviation %in notation 
   from \Cref{sec:preliminaries}, set 
   $
   L=r_w\textstyle\frac{\partial}{\partial z}-r_z\frac{\partial}{\partial w}.
   $
   Then   
   %We have to show that $\Hc{r}{L}{L}(z,w)\geq 0$ holds for
   %all $(z,w)\in b\Omega$ sufficiently close to $(0,0)$, and that this inequality is strict for $(z,w)\neq(0,0)$.
   %By definition,
   \begin{align*}
      \Hc{r}{L}{L}= r_{z\zbar}\abs{r_w}^2 - 2\RE[r_{z\wbar} r_w r_\zbar] + r_{w\wbar}\abs{r_z}^2. 
   \end{align*} 
   Computing the relevant terms, we obtain
  \begin{equation*} \begin{split}
    r_z(z,w) &=\textstyle \frac{1}{k} \zbar\abs{z}^{2k-2} - \frac{2}{k-1}\zbar\abs{z}^{2k-4}v 
              + \frac{1}{k-2}\zbar \abs{z}^{2k-6} v^2 + (2k-1)\zbar\abs{z}^{4k-4}, \\[0.5ex]
    r_w(z,w) &=\textstyle \frac{1}{2} 
             + i(\textstyle\frac{1}{(k-1)^2}\abs{z}^{2k-2} - \frac{1}{(k-2)^2}\abs{z}^{2k-4}v), \\[0.5ex]
    r_{z\zbar}(z,w) &=\textstyle \abs{z}^{2k-6}(\abs{z}^2-v)^2 + (2k-1)^2\abs{z}^{4k-4}, \\[0.5ex]
    r_{z\wbar}(z,w) &= -\textstyle\frac{i}{k-1}\zbar\abs{z}^{2k-4} + \frac{i}{k-2}\zbar\abs{z}^{2k-6}v, \\[0.5ex]
    r_{w\wbar}(z,w) &= \textstyle\frac{1}{2(k-2)^2}\abs{z}^{2k-4}. 
  \end{split} \end{equation*}
  These equations lead straightforwardly to the estimates
  \begin{equation*} \begin{split}
    (r_{z\zbar}\abs{r_w}^2)(z,w) &\ge \textstyle\frac{1}{4} \big[\abs{z}^{2k-6}(\abs{z}^2-v)^2 + \abs{z}^{4k-4}\big], \\
    -2\RE[r_{z\wbar} r_w r_\zbar](z,w) &= \abs{z}^{6k-14}\sum_{j=0}^4 \mathcal{O}(\abs{z}^{8-2j}v^j), \\
    (r_{w\wbar}\abs{r_z}^2)(z,w) &\ge 0. 
  \end{split} \end{equation*} 
  Setting $a \coloneqq \abs{z}^2-v$, it follows that for $\abs{z}$ and  $v$ sufficiently close to $0$
  \begin{equation*}\begin{split}
     \Hc{r}{L}{L}(z,w) 
     & \ge \textstyle\frac{1}{4}(\abs{z}^{2k-6}a^2 + \abs{z}^{4k-4}) + \abs{z}^{6k-14}\sum_{j=0}^4 \mathcal{O}(\abs{z}^{8-2j}a^j) \\
     & = \textstyle\frac{1}{4}\abs{z}^{2k-6}\big[(a^2 + \abs{z}^{2k+2}) +o(1)(a^2 + a\abs{z}^{k+1} + \abs{z}^{2k+2})\big] \\
     & \ge \textstyle\frac{1}{8}\abs{z}^{2k-6}(a^2 + \abs{z}^{2k+2}) .
  \end{split}\end{equation*}%
 This shows that $\Omega$ is pseudoconvex near $0$, and that the set of weakly pseudoconvex points of $\Omega$ has the form as described above. It is clear from (\ref{E:FiniteTypeCoordinates}) that $c_0 = 2k$.

  (ii) Assume, in order to get a contradiction, that $\rho \colon V \to \R$ is a $\mathcal{C}^2$-smooth local defining function for $\Omega$ near $0 \in b\Omega$ such that $\rho$ is plurisubharmonic on $b\Omega \cap V$. There exists a $\mathcal{C}^1$-smooth function $h \colon V \to \R$ such that $\rho = re^h$.
%and without loss of generality $h(0,0) =1$. 
Thus on $b\Omega\cap V$ one has (since $h \ast \delta_j \to h$ in the $\mathcal{C}^1$-norm for every Dirac sequence $\{\delta_j\}$)
\begin{equation} \begin{split} \label{equ:formrho}
\rho_{z\zbar} &= \big(r_{z\zbar} + 2\RE[r_zh_\zbar]\big)e^h , \\
\rho_{z\wbar} &= \big(r_{z\wbar} + r_zh_\wbar + r_\wbar h_z \big)e^h.
\end{split} \end{equation}
For $\varepsilon > 0$ sufficiently small, let $f \colon \D(0,\varepsilon) \to b\Omega \cap V$ be the smooth map 
\begin{align}\label{E:curveinboundary}
%f(\zeta) \coloneqq (\zeta,-c_k\abs{\zeta}^{2k} - \abs{\zeta}^{4k-2} + i\abs{\zeta}^2), \quad c_k \coloneqq \textstyle \frac{1}{k^2} - \frac{2}{(k-1)^2} + \frac{1}{(k-2)^2},
f(\zeta) \coloneqq (\zeta,u(|\zeta|) + i\abs{\zeta}^2),\, u(|\zeta|) \coloneqq \textstyle -\big(\frac{1}{k^2} - \frac{2}{(k-1)^2} + \frac{1}{(k-2)^2}\big)\abs{\zeta}^{2k} - \abs{\zeta}^{4k-2},
\end{align}
where $\D(0,\varepsilon) \coloneqq \{\zeta \in \C : \abs{\zeta} < \varepsilon\}$. We claim that 
\begin{align} \label{equ:orderhz}
(h_z \circ f)(\zeta) = O(\abs{\zeta}^{2k-3}).
\end{align}
Indeed, if not, then the number $m \in \N \cup \{0\}$ such that $(h_z \circ f)(\zeta) = O(\abs{\zeta}^m)$ but $(h_z \circ f)(\zeta) \neq O(\abs{\zeta}^{m+1})$ satisfies $m<2k-3$. Hence, in view of (\ref{equ:formrho}) and the computations in part (i), we see that $(\rho_{z\zbar} \circ f)(\zeta) = O(\abs{\zeta}^{2k-1+m})$ and $(\rho_{z\wbar} \circ f)(\zeta) \neq O(\abs{\zeta}^{m+1})$. Thus $((\rho_{z\zbar}\rho_{w\wbar}) \circ f)(\zeta) = O(\abs{\zeta}^{2k-1+m})$ and $(\abs{\rho_{z\wbar}}^2 \circ f)(\zeta) \neq O(\abs{\zeta}^{2m+2})$. In view of the inequality $2k-1+m>2m+2$, this contradicts the fact that $\rho$ is psh on $b\Omega$ near $0$, since this implies that $\left(\rho_{z\zbar}\rho_{w\wbar} - \abs{\rho_{z\wbar}}^2\right)\circ f \ge 0$.

From (\ref{equ:formrho}), (\ref{equ:orderhz}) and the computations in part (i), we conclude that
\begin{align*} %\label{equ:orderrho}
%(\rho_z \circ f)(\zeta) &= O(\abs{\zeta}^{2k-1}), \\
%(\rho_w \circ f)(\zeta) &= \textstyle\frac{1}{2}+ O(\abs{\zeta}^{2k-2}), \\
(\rho_{z\zbar} \circ f)(\zeta) &= O(\abs{\zeta}^{4k-4}), \\
(\rho_{z\wbar} \circ f)(\zeta) &= \textstyle i\mu_k\bar{\zeta}\abs{\zeta}^{2k-4} + \textstyle\frac{1}{2}(h_z\circ f)(\zeta) + O(\abs{\zeta}^{2k-2}), \quad \mu_k \coloneqq \frac{1}{(k-1)(k-2)}.
\end{align*}
Since $\left(\rho_{z\zbar}\rho_{w\wbar} - \abs{\rho_{z\wbar}}^2\right)\circ f \ge 0$, it follows that
\begin{equation} \label{equ:hz}
 (h_z \circ f)(\zeta) = -2i\mu_k\bar{\zeta}\abs{\zeta}^{2k-4} + o(\abs{\zeta}^{2k-3}).
\end{equation}
For given $\sigma > 0$, define the curve $\gamma_{\sigma}=\gamma=(\gamma_1,\gamma_2) \colon [0,2\pi] \to b\Omega$  
by $\gamma(t)\coloneqq f(\sigma e^{it})$.
Then, for $\sigma > 0$ sufficiently small, it follows from (\ref{equ:hz}) that
\begin{equation*} \begin{split}
 \int_\gamma dh 
% &= \int_0^{2\pi} h_z(\gamma(t)) \cdot \dot{\gamma_1}(t) \,dt + \int_0^{2\pi} h_\zbar(\gamma(t)) \cdot \overline{\dot{\gamma_1}(t)} \,dt \\
 &= 2\RE\int_0^{2\pi} h_z(\gamma(t)) \cdot \dot{\gamma_1}(t) \,dt \\
% &= \int_0^{2\pi}(-2i\mu_k\sigma^{2k-3}e^{-it} + o(\sigma^{2k-3})) \cdot (i\sigma e^{it}) \,dt + \int_0^{2\pi}(2i\mu_k\sigma^{2k-3}e^{it} + o(\sigma^{2k-3})) \cdot (-i\sigma e^{-it}) \,dt \\
 &= 2\RE\int_0^{2\pi}(-2i\mu_k\sigma^{2k-3}e^{-it} + o(\sigma^{2k-3})) \cdot (i\sigma e^{it}) \,dt \\
 &=  4\mu_k\int_0^{2\pi} (\sigma^{2k-2} + o(\sigma^{2k-2})) \,dt.
\end{split} \end{equation*}
Thus $\int_\gamma dh \neq 0$ whenever $\sigma>0$ is sufficiently small. This is a contradiction.
\end{proof}

\begin{theorem} \label{T:counterexampleglobal}
  For fixed $k\in\mathbb{N}$, $k\geq 3$, let
  $$r(z,w) \coloneqq u+\textstyle\frac{1}{k^2}|z|^{2k}-\frac{2}{(k-1)^2}|z|^{2k-2}v+\frac{1}{(k-2)^2}|z|^{2k-4}v^2+|z|^{4k-2}+ \abs{w}^2,$$
  and set 
  $$
  \Omega \coloneqq \{(z,w)\in\mathbb{C}^2 : r(z,w)<0 \}.
  $$
  Then the following assertions hold true.
  \begin{itemize}
    \item[(i)] $\Omega$ is a bounded domain with smooth real-analytic boundary.
    \item[(ii)] $\Omega$ is pseudoconvex. If $k=3$, then $0 \in b\Omega$ is the only weakly pseudoconvex boundary point of $\Omega$. If $k>3$, the set of weakly pseudonconvex boundary points of $\Omega$ is $b\Omega \cap (\{0\} \times \C)$. Moreover, $b\Omega$ is of finite type $c_0 = 2k$ at $0$.
    \item[(iii)] Any $\mathcal{C}^2$-smooth local defining function for $\Omega$ near $0 \in b\Omega$ fails to be plurisubharmonic on $b\Omega$ near $0$.
  \end{itemize}
\end{theorem}
\begin{proof}
   \noindent In order to show (iii), one may proceed exactly as in the proof of part (ii) of Theorem \ref{T:counterexample}, 
   after noting that there are two choices for $u(|\zeta|)$ in \eqref{E:curveinboundary} to be a solution to $r(\zeta,u(|\zeta|)+i|\zeta|^2)=0$. 
   While the proof of (i) is also straightforward, see below, 
   the difficulty in proving 
   Theorem \ref{T:counterexampleglobal} lies in showing that the introduction of the additional term $\abs{w}^2$ in the defining function $r$ turns $\Omega$ 
   into a globally pseudoconvex domain, which is subject to the precise properties described in (i) and (ii). 
   
   In the proofs of (i) and (ii), we repeatedly use the following fact: if $A,B,C \in \R$ and $A,C \ge 0$, then
   \begin{align} \label{E:strictsclc}
   \exists\; \varepsilon >0 \;\forall\, (\sigma,\tau) \in \R^2: 
   A\sigma^2 + B\sigma\tau + C\tau^2 \ge \varepsilon (\sigma^2+\tau^2) \;\;\Leftrightarrow\;\; 4AC-B^2 > 0.
   \end{align}

\noindent (i) It follows from \eqref{E:strictsclc}, with $\sigma = \abs{z}^2$ and $\tau = v$, that
%Since $\frac{4}{k^2(k-2)^2} - \frac{4}{(k-1)^4} > 0$ for $k \ge 3$, it follows from \eqref{E:strictsclc} that
   \begin{align*}
     \textstyle\frac{1}{k^2} |z|^{2k}&-\textstyle\frac{2}{(k-1)^2}|z|^{2k-2}v+\frac{1}{(k-2)^2}|z|^{2k-4}v^2 \geq 0.
   \end{align*}
   Thus, for every $(z,w)\in\bar{\Omega}$, one has $0 \geq r(z,w) \geq u+u^2$, and hence $u \in [-1,0]$. In particular, $u+u^2 \in [-\textstyle\frac{1}{4},0]$, so that
   \begin{align*}
     0&\geq r(z,w)\geq u+u^2+|z|^{4k-2}\geq-\textstyle\frac{1}{4}+|z|^{4k-2},\\
     0&\geq r(z,w)\geq u+u^2+v^2\geq-\textstyle\frac{1}{4}+v^2.
   \end{align*}
   Hence, $|z|^2\leq 4^{-\frac{1}{2k-1}}$ and $|v|\leq\frac{1}{2}$. This shows that $\Omega$ is bounded.
   
   To see that $b\Omega$ is smooth, we compute
   \begin{align*}
    r_z(z,w) &= \zbar \left(\textstyle \frac{1}{k} \abs{z}^{2k-2} - \frac{2}{k-1}\abs{z}^{2k-4}v 
              + \frac{1}{k-2} \abs{z}^{2k-6} v^2 + (2k-1)\abs{z}^{4k-4}\right), \\
    r_w(z,w) &=\textstyle (\frac{1}{2} + u)
             + i\left(\textstyle\frac{1}{(k-1)^2}\abs{z}^{2k-2} - \frac{1}{(k-2)^2}\abs{z}^{2k-4}v - v\right).
   \end{align*}
   It follows from \eqref{E:strictsclc} that $\frac{1}{k} \abs{z}^{2k-2} - \frac{2}{k-1}\abs{z}^{2k-4}v + \frac{1}{k-2} \abs{z}^{2k-6} v^2 \ge 0$, so $r_z(z,w) \neq 0$ whenever $z\neq0$. Moreover, if $(0,w) \in b\Omega$, then $u+u^2+v^2 = 0$. In this case, either $v=0$ and $u\in\{-1,0\}$, and thus $r_u(0,w)\neq 0$, or $v\neq 0$, and thus $r_v(0,w)\neq 0$.\smallskip
   
   (ii) Let $(z,w) \in b\Omega$. As before, write $a \coloneqq \abs{z}^2-v$. We consider two cases.\smallskip
   
   \noindent \textsc{Case 1}: $\abs{z} < \frac{1}{10}$ and $\abs{a} < \frac{1}{10}$. 
   Since then $\abs{v} < \frac{1}{5}$, it follows from $r(z,w) = 0$ that
   $$
   u+u^2 
   = \textstyle-\frac{1}{k^2}|z|^{2k}+\frac{2}{(k-1)^2}|z|^{2k-2}v-\frac{1}{(k-2)^2}|z|^{2k-4}v^2 -|z|^{4k-2}- v^2 
   > -\frac{3}{16}.
   $$
   Hence $u \notin [-\frac{3}{4},-\frac{1}{4}]$, and thus $\abs{r_u(z,w)} > \frac{1}{4}$. In particular,
   \begin{align} \label{E:pscestimate1}
   (r_{z\zbar}\abs{r_w}^2)(z,w) \ge \textstyle \frac{1}{16}\abs{z}^{2k-6}a^2 + \frac{25}{16}\abs{z}^{4k-4}.
   \end{align}
   Inserting the equation $v = \abs{z}^2 - a$ into the formulas for $r_z$, $r_w$ and $r_{z\bar{w}}$, we see that
   \begin{align*}
   r_z(z,w) &= \zbar\abs{z}^{2k-6} \left(\textstyle\frac{2}{k(k-1)(k-2)}\abs{z}^4 - \frac{2}{(k-1)(k-2)}\abs{z}^2a + \frac{1}{k-2}a^2 + (2k-1)\abs{z}^{2k+2} \right), \\
   r_w(z,w) &= (\textstyle\frac{1}{2}+u) + i \left(-\frac{2k-3}{(k-1)^2(k-2)^2}\abs{z}^{2k-2} + \frac{1}{(k-2)^2}\abs{z}^{2k-4}a - \abs{z}^2 + a\right), \\
   r_{z\bar{w}}(z,w) &= i\zbar\abs{z}^{2k-6}\left(\textstyle\frac{1}{(k-1)(k-2)}\abs{z}^2 - \frac{1}{k-2}a \right).
   \end{align*}
   Thus, since $\abs{z} < 1$ for $(z,w) \in b\Omega$, we obtain that, for every $k \ge 3$,
   \begin{align*}
   \abs{r_z(z,w)} &\le \abs{z}^{2k-5}\left((\abs{z}^2+\abs{a})^2 + (2k-1)\abs{z}^{2k+2}\right),\\
   \abs{r_v(z,w)} &\le 2(\abs{z}^2+\abs{a}),\\
   \abs{r_{z\bar{w}}(z,w)} &\le \abs{z}^{2k-5}(\abs{z}^2+\abs{a}).
   \end{align*}
   From $\abs{2\RE[r_{z\bar{w}}r_wr_{\zbar}]} = \abs{2\RE[r_{z\bar{w}}r_vr_{\zbar}]} \le 2\abs{r_{z\bar{w}}} \abs{r_v} \abs{r_z}$, it then follows that 
   \begin{align*}   
   \left|2\RE[r_{z\bar{w}}r_wr_{\zbar}]\right|(z,w) 
   &\le 4\abs{z}^{4k-10}(\abs{z}^2+\abs{a})^4 + (8k-4)\abs{z}^{6k-8}(\abs{z}^2+\abs{a})^2 \\
   &= \left(4\abs{z}^2 + (8k-4)\abs{z}^{2k-4}(\abs{z}^2+\abs{a})^2\right) \cdot \abs{z}^{4k-4} \\
   &\quad+ 16\abs{z}^{k+1} \cdot \abs{z}^{3k-5}\abs{a} \\
   &\quad+ 4\abs{z}^{2k-4}(6\abs{z}^4 + 4\abs{z}^2\abs{a} + \abs{a}^2) \cdot \abs{z}^{2k-6}\abs{a}^2.
   \end{align*}
   Since $\abs{z} < \frac{1}{10}$ and $\abs{a} < \frac{1}{10}$, this implies that, for every $k \ge 3$,
   \begin{align} \label{E:pscestimate2}
   \left|2\RE[r_{z\bar{w}}r_wr_{\zbar}]\right|(z,w) \le \textstyle\frac{11}{250}\abs{z}^{2k-6}\abs{a}^2 + \frac{4}{25}\abs{z}^{3k-5}\abs{a} + \frac{1}{2}\abs{z}^{4k-4}.
   \end{align}
   From $\Hc{r}{L}{L} \ge r_{z\zbar}\abs{r_w}^2 - \abs{2\RE[r_{z\bar{w}}r_wr_{\zbar}]}$, it follows with \eqref{E:pscestimate1} and \eqref{E:pscestimate2} that
   \begin{align*}
    \Hc{r}{L}{L}(z,w) \ge \abs{z}^{2k-6}\left(\textstyle\frac{37}{2000}\abs{a}^2 - \frac{4}{25}\abs{z}^{k+1}\abs{a} + \frac{17}{16}\abs{z}^{2k+2}\right).
   \end{align*}
   Thus, since $4 \cdot \frac{37}{2000} \cdot \frac{17}{16} - (\frac{4}{25})^2 > 0$, an application of \eqref{E:strictsclc} shows that 
   \begin{align*}
   \Hc{r}{L}{L}(z,w) \ge \varepsilon \abs{z}^{2k-6}(\abs{a}^2 + \abs{z}^{2k+2})
   \end{align*}
   for some constant $\varepsilon > 0$.\smallskip
   
   \noindent \textsc{Case 2}: $\abs{z} > \frac{1}{10}$ or $\abs{a} > \frac{1}{10}$. Set $D \coloneqq r_{z\zbar}r_{w\wbar} - \abs{r_{z\wbar}}^2$. Then
   \begin{align*}
   D(z,w) \ge \abs{z}^{2k-6}\abs{a}^2 \!+ (2k\!-\!1)^2\abs{z}^{4k-4} \!\!- \textstyle \frac{\abs{z}^{4k-10}}{(k-2)^2}\abs{a}^2 
   \!- \frac{2\abs{z}^{4k-8}}{(k-1)(k-2)^2}\abs{a} - \frac{\abs{z}^{4k-6}}{(k-1)^2(k-2)^2},
   \end{align*}
   so $D(z,w) \ge \abs{z}^{2k-6} d(z,w)$ with
   \begin{align*}
   d(z,w) \coloneqq \textstyle\left(1-\frac{\abs{z}^{2k-4}}{(k-2)^2}\right)\abs{a}^2 - 
   \frac{2\abs{z}^{2k-2}}{(k-1)(k-2)^2}\abs{a} + \left((2k-1)^2\abs{z}^{2k+2} - \frac{\abs{z}^{2k}}{(k-1)^2(k-2)^2}\right).
   \end{align*}
   We will show that $d(z,w) > 0$. Since, in the currently considered case, we have $r_{z\zbar}(z,w) \ge c\abs{z}^{2k-6}$ with some constant $c > 0$, this implies the claim.
   
   Assume first that $\abs{a} > \frac{1}{10}$ and $\abs{z}^2 \le \frac{1}{10}$. Then $d(z,w) \ge \frac{9}{10}a^2 - \frac{1}{100}a - \frac{1}{4000} > 0$. On the other hand, assume now that $\abs{z}^2 > \frac{1}{10}$. Then $d(z,w)>0$ provided that
   \begin{align*}
   4\textstyle\left(1-\frac{\abs{z}^{2k-4}}{(k-2)^2}\right)\left((2k-1)^2\abs{z}^{2k+2} - \frac{\abs{z}^{2k}}{(k-1)^2(k-2)^2}\right) -  \frac{4\abs{z}^{4k-4}}{(k-1)^2(k-2)^4} > 0,
   \end{align*}
   see \eqref{E:strictsclc}, and this inequality is satisfied if and only if
   \begin{align*}
   \textstyle\abs{z}^{2k-2} - (k-2)^2\abs{z}^2 + \frac{1}{(k-1)^2(2k-1)^2} < 0.
   \end{align*}
   But the left-hand side is negative for $\abs{z}^2 \in (\frac{1}{10},4^{-\frac{1}{2k-1}}] \coloneqq I_k$, since the function 
   $$
   f_k(t) \coloneqq t^{k-1} - (k-2)^2t + \textstyle\frac{1}{(k-1)^2(2k-1)^2}
   $$ is negative on $I_k$: Indeed, for $k=3$ a straightforward computation shows that $f_3 < 0$ on $(\frac{1}{2} - \frac{\sqrt{6}}{5}, \frac{1}{2} + \frac{\sqrt{6}}{5}) \supset I_3$, and for $k \ge 4$ note that $f_k(\frac{1}{10}) \le f_4(\frac{1}{10}) < 0$ and $f_k(1) < 0$, so that convexity of $f_k$ on $(0,\infty)$ implies $f_k<0$ on $(\frac{1}{10},1) \supset I_k$. \smallskip
\end{proof}

\begin{remark}
In the case of $k>3$, the defining function in \Cref{T:counterexample} can be modified in such a way that the origin is the only weakly pseudoconvex boundary point of the modified domain near the origin while maintaining all other properties of (i)-(ii). Similarly, in \Cref{T:counterexampleglobal}, 
the defining function can be adapted in such a way that the origin is the only weakly pseudoconvex boundary point of the thereby obtained domain while keeping all other properties of (i)-(iii).
\end{remark}

\section{Weakly pseudoconvex boundary points of type 4}\label{sec:stricttype4}
Let $\Omega \subset \C^2$ be a smoothly bounded, pseudoconvex domain, and $p_0\in b\Omega$. Assume that $\Omega$ is weakly pseudoconvex at $p_0$. Let $r \colon U \to \R$ be a smooth local defining function for $\Omega$ on an open neighborhood $U \subset \C^2$ of $p_0$, let $L$ be a nonvanishing holomorphic tangential vector field on $U$, and set
$$
\lambda \coloneqq \Hc{r}{L}{L}_{|_{b\Omega \cap U}}.
$$
Then $\lambda$ attains a local minimum at $p_0$. Since $L\bar{L}\lambda = \frac{1}{4}(XX + YY - i[X,Y])\lambda$, and since $[X,Y]$ is tangential to $b\Omega$, it follows that $(L\bar{L}\lambda)(p_0)$ is real. In this section we will show, in particular, that 
\begin{equation} \label{equ:type4condition}
(L\bar{L}\lambda)(p_0)\geq|(LL\lambda)(p_0)|.
\end{equation}

\begin{remark} \label{R:IndependenceLmabda}
Observe that (\ref{equ:type4condition}) is independent of the choice of $r$. Indeed, let $\rho \colon U \to \R$ be another smooth local defining function for $\Omega$. Then there exists a smooth function $h>0$ on $U$ such that $\rho = rh$, and one easily computes that $\lambda_\rho = h \lambda_r$, where $\lambda_\ast \coloneqq \Hc{\ast}{L}{L}_{|_{b\Omega \cap U}}$.
Since $\lambda_r$ attains a local minimum at $p_0$, all tangential derivatives of $\lambda_r$ at $p_0$ vanish. It thus follows that all second order tangential derivatives of $\lambda_r$ and $\lambda_\rho$ at $p_0$ differ by the same constant factor $c \coloneqq h(p_0) > 0$. In particular, 
\begin{align}
\label{equ:LLbarlambdatransform}
(L\bar{L}\lambda_\rho)(p_0) &= c (L\bar{L}\lambda_r)(p_0),\\
\label{equ:LLlambdatransform}
(LL\lambda_\rho)(p_0) &= c (LL\lambda_r)(p_0),
\end{align}
and thus (\ref{equ:type4condition}) is independent of the local defining function $r$.
%We will use 
%the notation $\lambda \coloneqq \lambda_r$ without explicit reference to the local defining function $r$, in case 
%the statement in question does not depend on the choice of $r$. 

Note further that (\ref{equ:type4condition}) is also independent of the choice of the holomorphic tangential vector field $L$.
%; in particular, $L$ does not have to be normalized as in (\ref{E:definitionLN}). 
Namely, for $L,L'$ nonvanishing holomorphic tangential vector fields on $U$, there exists a nonvanishing complex-valued function $h$ such that $L' = hL$ on $b\Omega \cap U$. Let $\lambda \coloneqq \Hc{r}{L}{L}_{|_{b\Omega \cap U}}$ and $\lambda' \coloneqq \Hc{r}{L'}{L'}_{|_{b\Omega \cap U}}$. Then, by similar arguments as above, one obtains that
\begin{align}
\label{equ:LLbarlambdatransform2}
(L'\bar{L}'\lambda')(p_0) &= \abs{c}^4 (L\bar{L}\lambda)(p_0),\\
\label{equ:LLlambdatransform2}
(L'L'\lambda')(p_0) &= c^2\abs{c}^2 (LL\lambda)(p_0),
\end{align}
where $c \coloneqq h(p_0)$.
\end{remark}

The following two Lemmata \ref{L:XXf}  and \ref {L:2ndderivativecondition} 
are used to derive \eqref{equ:type4condition} in Theorem~\ref{T:basictype4estimate}.

\begin{lemma}\label{L:XXf}
  Let $\Omega\subset\mathbb{R}^N$ be a smoothly bounded domain, $p_0\in b\Omega$, and $U$ an open neighborhood of $p_0$. 
  %Let $V$ be a smooth vector field on $b\Omega\cap U$ such that $V_p\in \Tr{p}{b\Omega}$ for all $p\in b\Omega\cap U$. 
  Let $V$ be a smooth vector field along $b\Omega\cap U$. 
  For $I\subset\mathbb{R}$ an open interval containing $0$, let $\gamma:I\longrightarrow b\Omega\cap U$ be a smooth curve 
  such that $\gamma(0)=p_0$ and $\dot{\gamma}(\tau)=V_{\gamma(\tau)}$ for every $\tau \in I$. Then
  \begin{align}\label{equ:2ndderivativeinboundary}
   (f\circ\gamma)''(0)=(VVf)(p_0).
  \end{align}  
  for every smooth function $f \colon b\Omega \cap U \to \R$.
\end{lemma}
\begin{proof}
  Without loss of generality, assume that $f$ and $V$ are defined on $U$. Then %with $\ddot{\gamma} = \sum\gamma_j''\frac{\partial}{\partial t_j}|_\gamma$,
  \begin{align*}
    \left(f\circ\gamma\right)''(0)&=\Hr{f}{V}{V}(p_0)+\left\langle \left(d f\right)(p_0),\ddot{\gamma}(0)\right\rangle.
  \end{align*}
  On the other hand, we see from (\ref{equ:XYf}) that
  \begin{align*}
    (VVf)(p_0) = \Hr{f}{V}{V}(p_0) + \left\langle \left(d f\right)(p_0),\nabla_VV(p_0)\right\rangle.
  \end{align*}
  Thus, it suffices to show that $\ddot{\gamma}(0) = \nabla_VV(p_0)$. But this is clear, 
  since $V$ is an extension of $\dot\gamma$, and since $\ddot \gamma$ is the covariant 
  derivative of $\dot\gamma$ with respect to $\nabla$. We can also compute this directly 
  as follows. For every $\nu \in \{1, \ldots, N\}$ and $\tau \in I$,
  \begin{align*}
    \gamma_\nu''(\tau) 
    = (V_\nu \circ \gamma)' (\tau)
   & = \sum_{\mu=1}^N\frac{\partial V_\nu}{\partial t_\mu}(\gamma(\tau))\gamma_\mu'(\tau) \\
    &= \sum_{\mu=1}^N \left(V_\mu\frac{\partial V_\nu}{\partial t_\mu}\right)(\gamma(\tau)) = V(V_\nu)(\gamma(\tau)).
  \end{align*}
\end{proof}

\begin{lemma}\label{L:2ndderivativecondition}
  Let $\Omega\subset\mathbb{R}^N$ be a smoothly bounded domain, $p_0\in b\Omega$, and 
  $U$ an open neighborhood of $p_0$.
  Let $f \colon b\Omega \cap U \to \R$ be smooth such that $f$ attains a local minimum at $p_0$. 
  Suppose that $V^1, \ldots, V^{N-1}$ are smooth vector fields along $b\Omega\cap U$ such that 
  $\smash{\left(V^1 , \ldots, V^{N-1}\right)_{p_0}}$ is a basis for $\Tr{p_0}{b\Omega}$. 
  Then the matrix
  $$
  \begin{pmatrix}
    V^1V^1f & \cdots & V^1V^{N-1}f  \\
    \vdots  &        & \vdots   \\
    V^{N-1}V^1f & \cdots & V^{N-1}V^{N-1} f
  \end{pmatrix}(p_0)
  $$
  is symmetric and positive semi-definite.
\end{lemma}
\begin{proof}
 Consider the function $q \colon \Tr{p_0}{b\Omega} \to \R$ given by
 $$
 q(V_{p_0}) \coloneqq (VV f)(p_0),
 $$
 where $V$ is any extension of $V_{p_0}$ to a vector field along $b\Omega \cap U$ 
 such that $V_p \in \Tr{p}{b\Omega}$ for every $p \in b\Omega \cap U$. 
 This is well-defined. To wit, if $\bar\nabla$ is any linear connection on $b\Omega$, then 
 $$
 VVf = \bar{Q}_f^{\scriptscriptstyle\mathbb{R}}(V,V) + \left(\bar\nabla_VV\right)f,
 $$
 where $\bar{Q}_f^{\scriptscriptstyle\mathbb{R}} = \bar\nabla^2f$ denotes the 
 covariant Hessian. Since $p_0$ is a local minimum of $f$, the derivative $(\bar\nabla_VV)f$ 
 vanishes at $p_0$ independently of the above choice of $V$. 
 Furthermore, since $\bar{Q}_f^{\scriptscriptstyle\mathbb{R}}$ is a tensor, 
 the value $\bar{Q}_f^{\scriptscriptstyle\mathbb{R}}(V,V)(p_0)$ depends only on $V_{p_0}$. 
 %For the convenience of the reader, we point out that the well-definedness of $q$ can also be 
 %concluded in a more elementary way as follows. 
 %Differentiating $r \circ \gamma$, where $r$ is any smooth local defining function for $\Omega$ near $p_0$, and 
 %$\gamma \colon I \to b\Omega$ is a smooth curve on an open interval containing $0$ such that 
 %$\gamma(0) = p_0$ and $\dot{\gamma}(0) = V_{p_0}$, shows that 
 %$$\left\langle \left(d r\right)(p_0),\ddot{\gamma}(0)\right\rangle = -\Hr{r}{V}{V}(p_0),$$ 
 %i.e., the normal component of $\ddot{\gamma}(0)$ depends only on $\dot\gamma(0)$. Moreover, 
 %if an extension of $f$ to an open neighborhood of $p_0$ is fixed, then 
 %since $f_{|_{b\Omega}}$ attains a local minimum at $p_0$, it follows that the gradient of $f$  
 %at $p_0$ is orthogonal to $b\Omega$, so 
 %$\left\langle \left(d f\right)(p_0),\ddot{\gamma}(0)\right\rangle$ depends only on the normal 
 %component of $\ddot\gamma(0)$. We conclude that the number  $(f \circ \gamma)''(0) = 
 %\Hr{f}{V}{V}(p_0)+\left\langle \left(d f\right)(p_0),\ddot{\gamma}(0)\right\rangle$ 
 %depends only on $\dot{\gamma}(0)$. 
 %The claim thus follows from \Cref{L:XXf}.

 It follows from (\ref{equ:2ndderivativeinboundary}) and an application of the Picard--Lindelöf 
 theorem, that $q \ge 0$. Thus the associated symmetric bilinear form 
 $B \colon \Tr{p_0}{b\Omega} \times \Tr{p_0}{b\Omega} \to \R$,
 $$
   B\!\left(V_{p_0},W_{p_0}\right) 
   \coloneqq \textstyle\frac{1}{2}\left(q(V_{p_0}+W_{p_0})-q(V_{p_0})-q(W_{p_0})\right),
 $$
 is positive semi-definite. Moreover, note that
 $$
   B\!\left(V^j_{p_0},V^k_{p_0}\right) 
   = \textstyle\left(\frac{1}{2}V^jV^kf + \frac{1}{2}V^kV^j f\right)(p_0) 
   = \left(V^j V^k f + \frac{1}{2}[V^k,V^j]f\right)(p_0).
 $$  
 Since the vector field $[V^k,V^j]$ is tangential to $b\Omega$, and since $f$ attains a local minimum at $p_0$, 
 it follows that $[V^k,V^j]f$ vanishes at $p_0$. Therefore,
 \begin{align*}
     B\left(V^j_{p_0},V^k_{p_0}\right)=\left(V^j V^kf\right)(p_0),
 \end{align*}
 %and the associated matrix is symmetric and positive semi-definite.
 which proves the claim.
\end{proof}

\begin{theorem}\label{T:basictype4estimate}
  Let $\Omega\subset\mathbb{C}^2$ be a smoothly bounded, pseudoconvex domain. Assume that $\Omega$ is
  weakly pseudoconvex at $p_0 \in b\Omega$. Then
  \begin{align}\label{E:basictype4estimate}
    \left(L\bar{L}\lambda\right)(p_0)\geq |\left(LL\lambda\right)(p_0)|.
  \end{align}
\end{theorem}
\begin{proof}
   Let $X$, $Y$ and $T$ be the real vector fields such that
  \begin{align*}
    L=\textstyle\frac{1}{2}\left(X+iY\right)\quad\text{and}\quad N=\textstyle\frac{1}{2}\left(\nu+iT\right),
  \end{align*} 
  and recall that $X_p$, $Y_p$, and $T_p$ form a basis of $\Tr{p}{b\Omega}$ for all $p\in b\Omega$, see \Cref{sec:CanonicalVectorFields}.
%  Let $(X,Y,T)$ be a local frame for $T(b\Omega)$ on an open neighborhood $U$ of $p_0$ such that $L_{|_{b\Omega}} = \frac{1}{2}(X+iY)$.
  Since $\lambda$ attains a local minimum at $p_0$, it follows from Lemma \ref{L:2ndderivativecondition} that
  $$
  \begin{pmatrix}
    XX\lambda & XY\lambda& XT\lambda  \\
    YX\lambda  & YY\lambda    & YT\lambda   \\
    TX\lambda & TY\lambda & TT \lambda
  \end{pmatrix}(p_0)
  $$
  is symmetric and positive semi-definite. In particular,
   \begin{align*}%\label{E:XYpositivity1}
  \begin{pmatrix}
    XX\lambda & XY\lambda   \\
    YX\lambda  & YY\lambda   
  \end{pmatrix}(p_0)
  \end{align*}
  is symmetric and positive-semidefinite, i.e.,
  \begin{align*}%\label{E:XYpositivity2}
    \left(XX\lambda\cdot YY\lambda-(XY\lambda)^2\right)(p_0)\geq 0.
  \end{align*}  
  But a straightforward computation shows that 
  \begin{align*}
    &\left(L\bar{L}\lambda\right)^{2}-|LL\lambda|^2\\
    &=\textstyle\frac{1}{16} \bigl(\left(X+iY\right)\left(X-iY\right)\lambda\bigr)^2
       - \textstyle\frac{1}{16} \bigl|(X+iY)(X+iY)\lambda \bigr|^2\\
    &=\textstyle\frac{1}{16} \bigl(XX\lambda+YY\lambda\bigr)^2 
       - \textstyle\frac{1}{16}\bigl|XX\lambda-YY\lambda+2iXY\lambda\bigr|^2\\
    &=\textstyle\frac{1}{4}\left(XX\lambda\cdot YY\lambda-(XY\lambda)^2 \right).
  \end{align*}
  This proves the claim.
\end{proof}

\begin{definition} \label{def:stricttype4}
  Let $\Omega\subset\mathbb{C}^2$ be a smoothly bounded, pseudoconvex domain. We say that $\Omega$ is of strict type 4 at $p_0 \in b\Omega$, if $\Omega$ is weakly pseudoconvex at $p_0$ and 
    \begin{align}\label{E:strictbasictype4estimate}
    \left(L\bar{L}\lambda\right)(p_0) > \abs{\left(LL\lambda\right)(p_0)}.
  \end{align} 
  If \eqref{E:strictbasictype4estimate} does not hold for  $p_0\in b\Omega$ with $c_{p_0}=4$, then we say that $\Omega$ is of weak type $4$ at $p_0$.
\end{definition}

%In the next section, we will show the existence of local plurisubharmonic defining functions on $b\Omega$ near $p_0$, provided that the strict inequality
%\begin{align}\label{E:strongbasictype4estimate}
%  \left(\bar{L}L\lambda\right)(p_0) > |\left(LL\lambda\right)(p_0)|
%\end{align}
%holds true. 
From (\ref{equ:LLbarlambdatransform}) and (\ref{equ:LLlambdatransform}), we see that (\ref{E:strictbasictype4estimate}) is independent of the choice of a local defining function. Therefore, it describes a property of the domain $\Omega$ at $p_0$. In the next lemma, it is shown that this property is invariant under biholomorphic transformations. 

\begin{lemma}\label{thm:strictestimatetransform}
Let $\Omega',\Omega \subset \C^2$ be smoothly bounded domains such that $\Omega'$ and $\Omega$ are pseudoconvex near $p_0' \in b\Omega'$ and $p_0 \in b\Omega$, respectively. Let $\Phi \colon U' \to U$ be a biholomorphic map from an open neighborhood $U' \subset \C^2$ of $p_0'$ to an open neighborhood $U \subset \C^2$ of $p_0$ such that $\Phi(\Omega' \cap U') = \Omega \cap U$ and $\Phi(p_0') = p_0$. Then $p_0'$ is of strict type $4$ for $\Omega'$ if and only if $p_0$ is of strict type $4$ for $\Omega$.
\end{lemma}
\begin{proof}
Let $r \colon U \to \R$ be a smooth local defining function for $\Omega$ near $p_0$. Let $L'$ be a nonvanishing holomorphic tangential vector field on $U'$, and let $L \coloneqq \Phi_\ast L'$ be the pushforward of $L'$. Then define
$$
\lambda' = \Hc{r\circ\Phi}{L'}{L'}_{|_{b\Omega' \cap U'}} \quad\text{and}\quad \lambda = \Hc{r}{L}{L}_{|_{b\Omega \cap U}},
$$
and observe that, by the usual transformation law, one has $\lambda' = \lambda \circ \Phi$. Since, by definition, $L'(f \circ \Phi) = Lf$ for every smooth function $f \colon U \to \C$, it thus follows that
\begin{align} \label{equ:lambdatransform}
 \left(L'\bar{L}'\lambda'\right)(p_0') =  \left( L\bar{L}\lambda \right)(p_0) \quad\text{and}\quad
 \left(L'L'\lambda'\right)(p_0') = \left( LL\lambda \right)(p_0).
\end{align}
In view of \Cref{R:IndependenceLmabda}, this proves the claim.
\end{proof}

In view of \Cref{thm:strictestimatetransform}, it is meaningful to look for local holomorphic coordinates around $p_0$, in which the condition (\ref{E:strictbasictype4estimate}) takes a particularly simple form. The next result shows how this can be achieved.

\begin{proposition} \label{thm:stricttype4incoordinates}
Let $\Omega\subset\mathbb{C}^2$ be a smoothly bounded, pseudoconvex domain. Assume that $0\in b\Omega$ is a point of weak pseudoconvexity, and let $r$ be a smooth local defining function for $\Omega$ near $0$. If $r_z(0) = r_{zz}(0) = 0$, then   
\begin{align} \label{E:strictbasictype4estimateagain}
\left(L\bar{L}\lambda\right)(0)> |\left(LL\lambda\right)(0)|
\end{align}
if and only if there exists a constant $\varepsilon>0$ such that
\begin{align} \label{equ:stricttype4incoordinates}
 \Hc{r}{L}{L}(z,0) \ge \varepsilon|z|^2 + o(\abs{z}^2).
\end{align}
\end{proposition}
\begin{proof}
Let $U \subset \C^2$ denote the domain of definition of $r$, and define $\Lambda \colon U \to \R$ by $\Lambda \coloneqq \Hc{r}{L}{L}$. We will show that both (\ref{E:strictbasictype4estimateagain}) and (\ref{equ:stricttype4incoordinates}) are equivalent to the condition that the matrix
$$
A \coloneqq \begin{pmatrix} \Lambda_{xx} & \Lambda_{xy} \\ \Lambda_{yx} & \Lambda_{yy} \end{pmatrix}(0)
$$
is positive definite, where $z = x+iy$ with $x,y \in \R$. This proves the claim.

Observe first that $\Lambda_{|_{b\Omega \cap U}} = \lambda$. In particular, $\Lambda(0) = 0$. Moreover, since $\lambda$ attains a local minimum at $0$, it follows that $d\Lambda$ vanishes on $\Tr{0}{b\Omega}$, and thus $\Lambda_x(0) = \Lambda_y(0) = 0$.
Hence
$$
 \Lambda(z,0) = \textstyle\frac{1}{2}\displaystyle 
 \begin{pmatrix} x & y \end{pmatrix} A \begin{pmatrix} x \\ y \end{pmatrix} + o(\abs{z}^2).
$$
This shows that (\ref{equ:stricttype4incoordinates}) holds true if and only if $A>0$. 

On the other hand, we claim that $r_z(0) = r_{zz}(0) = 0$ implies that 
\begin{equation}\begin{split} \label{equ:LLLambda}
(L\bar{L}\lambda)(0) = \Lambda_{z\bar{z}}(0), \quad
(LL\lambda)(0) = \Lambda_{zz}(0).
\end{split}\end{equation}
From this, we immediately obtain that
\begin{align*}
(L\bar{L}\lambda)(0) &= \textstyle\frac{1}{4}(\Lambda_{xx} + \Lambda_{yy})(0),\\
\left((L\bar{L}\Lambda)^2 - \abs{LL\Lambda}^2\right)(0) &= \textstyle\frac{1}{4}\left(\Lambda_{xx}\Lambda_{yy}-\Lambda_{xy}^2\right)(0).
\end{align*}
Since a real $2 \times 2$ matrix is positive definite if and only if both its trace and determinant are positive, it follows that $(\ref{E:strictbasictype4estimateagain})$ holds true if and only if $A > 0$.

In order to see (\ref{equ:LLLambda}) we first note that, in view of \Cref{R:IndependenceLmabda}, and after possibly shrinking $U$, we can assume without loss of generality that $L=L_r$. Recall that
\begin{align*}
L\bar{L}\lambda &= \Hc{\Lambda}{L}{L} + (\nabla_L\bar{L})\Lambda,\\
LL\lambda &= \Qc{\Lambda}{L}{L} + (\nabla_LL)\Lambda.
\end{align*}
Note that the vectors $(\nabla_L\bar{L})(0)$ and $(\nabla_LL)(0)$ are tangential to $b\Omega$. Indeed, straightforward computations show that, for $L=L_r$ and $N=N_r$,
\begin{align*}
\nabla_L\bar{L}  &= \frac{1}{\abs{r_z}^2+\abs{r_w}^2}\Big(\Hc{r}{L}{N} \bar{L} - \Hc{r}{L}{L} \bar{N}\Big),\\
\nabla_LL  &= \frac{1}{\abs{r_z}^2+\abs{r_w}^2}\Big(\Qc{r}{L}{N} L - \Qc{r}{L}{L} N\Big).
\end{align*}
Moreover, $\Hc{r}{L}{L}(0) = 0$, since $\Omega$ is weakly pseudoconvex at $0$, and $\Qc{r}{L}{L}(0) = 0$, since $r_z(0) = r_{zz}(0) = 0$, which proves the claim. It follows that both derivatives $(\nabla_{\bar{L}}L)\Lambda$ and $(\nabla_LL)\Lambda$ vanish at $0$, since $\lambda$ attains a local minimum there. Since the condition $r_z(0) = 0$ implies that $L_0 = (\partial_z)_0$, the claim follows. (For an alternative proof of (\ref{equ:LLLambda}) under slightly stronger conditions, see \cite[Lemma 3.23]{Kohn72}.)
\end{proof}

%\begin{remark}
%In the proof of \Cref{thm:stricttype4incoordinates}, the condition $r_{zz}(0) = 0$ was needed in order to conclude that $(\nabla_LL)\Lambda$ vanishes at $0$. Observe that, if $r$ is of the form 
%$$r(z,w) = \RE w + R(z,\IM w)$$ for some smooth function $R$ that vanishes at $0$ to second order, then $\Lambda = \Hc{r}{L}{L}$ does not depend on $\RE w$, and therefore 
%\end{remark}

\begin{remark}
Let $\Omega \subset \C^2$ be a smoothly bounded, pseudoconvex domain. Then $\Omega$ is of strict type $4$ at a weakly pseudoconvex point $p_0 \in b\Omega$ in the sense of Kohn, if 
\begin{align} \label{equ:KohnType4}
 \textstyle \frac{1}{3}\RE[(LL\lambda)(p_0)] + \frac{1}{4} (L\bar{L}\lambda)(p_0) > 0 
\end{align}
holds for all nonvanishing tangential holomorphic vector fields $L$ near $p_0$, see \cite[Definition 2.16]{Kohn72} in the case $m=3$.

If $L' \coloneqq hL$ for some smooth complex-valued function $h$ near $p_0$, then one has $\frac{1}{3}\RE[(L'L'\lambda)(p_0)] + \frac{1}{4} (L'\bar{L}'\lambda)(p_0) = \frac{1}{3}\RE[h(p_0)^2(LL\lambda)(p_0)] + \frac{1}{4} \abs{h(p_0)}^2(L\bar{L}\lambda)(p_0)$. From this, one easily sees that $\Omega$ is of strict type 4 at $p_0$ in the sense of Kohn if and only if 
\begin{align} \label{equ:KohnType4b}
\left(L\bar{L}\lambda\right)(p_0) > \textstyle\frac{4}{3} \abs{\left(LL\lambda\right)(p_0)}
\end{align}
for some, and then every, nonvanishing tangential holomorphic vector field $L$ near $p_0$. The condition (\ref{equ:KohnType4b}) is invariant under biholomorphic transformations by (\ref{equ:lambdatransform}).

Assume that coordinates are chosen in such a way that $p_0 = 0$ and $(\partial_z^jr)(0) = 0$ for $j \in\{1, \ldots, 4\}$. Then (\ref{equ:KohnType4b}) holds true if and only if there exists a constant $\varepsilon>0$ such that
\begin{align} \label{equ:strictKohntype4incoordinates}
 r(z,0) \ge \varepsilon|z|^4 + o(\abs{z}^4).
\end{align}
Indeed, since $(\partial_z^j\partial_{\bar{z}}^kr)(0) = 0$ for $j+k \le 2$, it follows that $0 = (L\lambda)(0) = \lambda_z(0) = \Lambda_z(0) = (\partial_z^2\partial_{\bar{z}}r)(0)$, where $\Lambda \coloneqq \Hc{r}{L}{L}$ and without loss of generality $L = L_r$. Thus
$
r(z,0) = \textstyle\frac{1}{3}\RE[(\partial_z^3\partial_{\bar{z}}r)(0)z^3\bar{z}] + \frac{1}{4}(\partial_z^2\partial_{\bar{z}}^2r)(0)\abs{z}^4 + o(\abs{z})^4.
$
Since $(\partial_z^j\partial_{\bar{z}}^kr)(0) = 0$ for $j+k \le 3$, it follows with (\ref{equ:LLLambda}) that $(LL\lambda)(0) = (\partial_z^3\partial_{\bar{z}}r)(0)$ and $(\bar{L}L\lambda)(0) = (\partial_z^2\partial_{\bar{z}}^2r)(0)$. (The characterization of points of strict type 4 in the sense of Kohn by means of (\ref{equ:strictKohntype4incoordinates}) is already implicitly contained in Kohn's original paper, see formulas (3.8) and (3.12) in \cite{Kohn72}. See also \cite[Theorem 3.3]{Bloom78}.)

In view of (\ref{equ:KohnType4b}), the definition of strict type 4 given in \Cref{def:stricttype4} is more general than the notion of strict type 4 in the sense of Kohn. In particular, if $\Omega$ is of strict type 4 at $p_0$ in the sense of Kohn, then it is of strict type 4 at $p_0$ in the sense of \Cref{def:stricttype4}.
 
Lastly, consider the following example. Let $r \colon \C^2 \to \R$ be given by 
$$
r(z,w) \coloneqq \RE w + \RE[az^3\bar{z}] + \abs{z}^4.
$$
Then $\Omega \coloneqq \{r<0\}$ is pseudoconvex at $0 \in b\Omega$ iff $\abs{a} \le \frac{4}{3}$. Moreover, by checking the conditions (\ref{equ:stricttype4incoordinates}) and (\ref{equ:strictKohntype4incoordinates}), one easily sees that $0$ is of strict type 4 in the sense of \Cref{def:stricttype4} iff $\abs{a} < \frac{4}{3}$, and $0$ is of strict type 4 in the sense of Kohn iff $\abs{a} < 1$.
%\begin{align*}
%\abs{a} &\le \textstyle\frac{4}{3} &&\Leftrightarrow&& \text{$\Omega$ is weakly pseudoconvex at $p_0$}\\
%\abs{a} &< \textstyle\frac{4}{3} &&\Leftrightarrow&& \text{$p_0$ is of strict type 4}\\
%\abs{a} &< 1 &&\Leftrightarrow&& \text{$p_0$ is of strict type 4 in the sense of Kohn}\\
%\end{align*}

\end{remark}

\section{Plurisubharmonicity on the boundary} \label{sec:pshboudary}

The main goal of this section is to prove the following theorem.
%The main result of this section is that near a boundary point of strict type 4, the domain admits a 
%smooth local defining function which is plurisubharmonic on the boundary near that point.

\begin{theorem}\label{T:type4pshonboundary}
  Let $\Omega\subset\mathbb{C}^2$ be a smoothly bounded, pseudoconvex domain, and let $p_0\in b\Omega$ 
  be a point of weak pseudoconvexity for $\Omega$. If $p_0$ is of strict type $4$,
  %\begin{align}\label{E:strictbasictype4estimate2}
  %  \bar{L}L\lambda(p_0)> |LL\lambda (p_0)|,
  %\end{align}  
  then $\Omega$ admits a smooth local defining function which is plurisubharmonic on $b\Omega$ near $p_0$. 
\end{theorem}

We use the following basic lemma to show \Cref{T:type4pshonboundary}. 

\begin{proposition}\label{L:pshonboundaryiff}
 Let $\Omega\subset\mathbb{C}^2$ be a smoothly bounded, pseudoconvex domain, $p_0\in b\Omega$. 
 Then $\Omega$ admits a smooth local defining function which is plurisubharmonic on $b\Omega$ near $p_0$ 
 if and only if there exist an open neighborhood $U$ of $p_0$ and a smooth local defining function $\rho$ for $\Omega$ 
 on $U$ such that
 \begin{align}\label{E:pshonboundary}
  \left|\Hc{\rho}{L}{N}\right|^2=\mathcal{O}(\Hc{\rho}{L}{L})\text{ on } b\Omega\cap U.
 \end{align}
 In fact, \eqref{E:pshonboundary} holds true for every smooth local defining function $\rho \colon U \to \R$ 
 that is plurisubharmonic on $b\Omega \cap U$.
\end{proposition}

\begin{proof}
 If $\rho \colon U \to \R$ is plurisubharmonic on $b\Omega \cap U$, then
 \begin{align*}
  \Hcsymbolcal{\rho} =
   \begin{pmatrix}
    \Hc{\rho}{L}{L} & \Hc{\rho}{L}{N} \\
    \Hc{\rho}{N}{L} & \Hc{\rho}{N}{N}
   \end{pmatrix}
 \end{align*}
 is positive semi-definite at every point $p \in b\Omega \cap U$. In particular, $\det \Hcsymbolcal{\rho} \ge 0$ 
 on $b\Omega \cap U$, which implies \eqref{E:pshonboundary}.
 
 On the other hand, suppose that \eqref{E:pshonboundary} holds for some smooth local defining function 
 $\rho \colon U \to \R$ of $\Omega$ near $p_0$. Let $\chi \colon \R \to \R$ be a smooth function such 
 that $\chi(0) = 0$, $\chi'(0) = 1$, and write $\chi''(0) \eqqcolon C$. If $\hat\rho \coloneqq \chi \circ \rho$, then we get
 \begin{align*}
  \Hcsymbolcal{\hat\rho} =
  \begin{pmatrix}
   \Hc{\rho}{L}{L} & \Hc{\rho}{L}{N}  \\
   \Hc{\rho}{N}{L} & \Hc{\rho}{N}{N} + C \abs{N\rho}^2
  \end{pmatrix} 
  \quad\text{on } b\Omega \cap U.  
 \end{align*}
Let $V \Subset U$ be another open neighborhood of $p_0$. Note that $\abs{N\rho} = \frac{1}{2}\abs{d\rho} > c > 0$ on $b\Omega \cap V$ for some $c \in \R$. Thus, if $C>0$ is sufficiently large, it follows that $\TRACE\Hcsymbolcal{\hat\rho} > 0$ and $\det \Hcsymbolcal{\hat\rho} \ge 0$ on $b\Omega \cap V$, where for the second inequality we use assumption \eqref{E:pshonboundary}. This means that $\Hcsymbolcal{\hat\rho} \ge 0$ on $b\Omega \cap V$, i.e., $\hat\rho$ is plurisubharmonic on $b\Omega \cap V$.
\end{proof}

\begin{remark} \label{R:pshonboundaryiff}
The proof of \Cref{L:pshonboundaryiff} shows, in particular, the following:
%The proof of \Cref{L:pshonboundaryiff} shows, in particular, that 
if \eqref{E:pshonboundary} holds true, then 
for every $V \Subset U$ there exists a smooth local defining function for $\Omega$ on $U$ that is plurisubharmonic on $b\Omega \cap V$. 
%In fact, for $K>0$ large enough, the function $\hat{\rho} \coloneqq \rho + K\rho^2$ has the desired properties.
%More precisely, for every $V \Subset U$ there exists a constant $K>0$ sufficiently large such that the function $\rho + K\rho^2$ is plurisubharmonic on $V$.
\end{remark}

Note that \eqref{E:pshonboundary} may be reformulated as
$$\left|\Hc{\rho}{L}{N}\right|=\mathcal{O}(\sqrt{\lambda})\text{ on } b\Omega\cap U,$$
where $\lambda = \Hc{r}{L}{L}_{|_{b\Omega\cap U}}$ for any smooth local defining function $r \colon U \to \R$ for $\Omega$ and any nonvanishing holomorphic tangential vector field $L$ on $U$, see Remark \ref{R:Landau}.

\begin{proof}[Proof of Theorem \ref{T:type4pshonboundary}]
  Let $r \colon U \to \R$ be a smooth local defining function for $\Omega$ near $p_0$, and set 
  $\lambda \coloneqq \Hc{r}{L}{L}_{|_{b\Omega\cap U}}$.
  Since $\Omega$ is of strict type $4$ at $p_0$, after possibly shrinking $U$, we may assume that 
  $B \coloneqq |L\bar{L}\lambda|^2-|LL\lambda|^2 > 0$ on $U$. We claim that, for every smooth function 
  $F \colon U \to \C$, there exists a smooth function $h \colon U \to \R$ such that
  \begin{align}\label{E:claimforhgeneral}
    Lh = F + \mathcal{O}(\sqrt{\lambda}) \quad\text{on } b\Omega \cap U.
  \end{align}
  Indeed, define smooth functions $A_1, A_2 \colon U \to \R$ by 
  \begin{align*}
    A_1 \coloneqq 2\re\left((\bar{L}L\lambda-LL\lambda)\bar{L}\lambda \right) \quad\text{and}\quad
    A_2 \coloneqq 2\re\left(i(\bar{L}L\lambda+LL\lambda)\bar{L}\lambda \right),
  \end{align*}
  and set
  \begin{align}\label{E:definitionh}
    h \coloneqq \RE (F) \frac{A_1}{B} + \IM (F) \frac{A_2}{B}.
  \end{align}
  By \Cref{thm:lHospital}, both $L\lambda$ and $\bar{L}\lambda$ are of class 
  $\mathcal{O}(\sqrt{\lambda})$. Thus
  \begin{align*}
    Lh = \RE (F) \frac{LA_1}{B} + \IM (F) \frac{LA_2}{B} + \mathcal{O}(\sqrt{\lambda}) 
    \quad\text{on } b\Omega \cap U,
  \end{align*}
  and, again on $b\Omega \cap U$,
  \begin{align*}
  LA_1 &= (\bar{L}L\lambda-LL\lambda)L\bar{L}\lambda+(L\bar{L}\lambda-\bar{L}\bar{L}\lambda)LL\lambda
          +\mathcal{O}(\sqrt{\lambda}) = B +\mathcal{O}(\sqrt{\lambda}), \\
  LA_2 &= i(\bar{L}L\lambda+LL\lambda)L\bar{L}\lambda-i(L\bar{L}\lambda+\bar{L}\bar{L}\lambda)LL\lambda
          +\mathcal{O}(\sqrt{\lambda}) = iB +\mathcal{O}(\sqrt{\lambda}).
  \end{align*}
  Hence $h$ is a solution to \eqref{E:claimforhgeneral}.
 
  Now let $h \colon U \to \R$ be an arbitrary smooth function, and set $\rho \coloneqq re^h$. Then
  \begin{align*}
    \Hc{\rho}{L}{N} = e^h\left(\Hc{r}{L}{N}+ Lh\cdot \bar{N}r \right) \quad\text{on } b\Omega \cap U.
  \end{align*}
  Since $\bar{N}r = \frac{1}{2}\abs{dr} \neq 0$ on $b\Omega \cap U$, it follows that \eqref{E:pshonboundary} 
  is satisfied if $h$ is the solution for \eqref{E:claimforhgeneral} with 
  $F \coloneqq -\frac{2}{\abs{dr}}\Hc{r}{L}{N}$.
  Thus, the claim follows from \Cref{L:pshonboundaryiff}.
\end{proof}

\begin{remark} The functions $A_1$, $A_2$, and $B$ in the proof of \Cref{T:type4pshonboundary} depend on the given smooth local defining function $r \colon U \to \R$. However, the function $h$ defined in \eqref{E:definitionh} solves \eqref{E:claimforhgeneral} for any choice of $r$. 
%In particular, recall that the class of smooth function $f \colon b\Omega \to \C$ satisfying $\abs{f} = \mathcal{O}(\sqrt{\lambda})$ on $b\Omega \cap U$ is independent of the local defining function $r$ used to define $\lambda$.
\end{remark}

 A global version of \Cref{T:type4pshonboundary} easily follows.

\begin{corollary}
  Let $\Omega\Subset\mathbb{C}^2$ be a smoothly bounded, pseudoconvex domain. Assume that all weakly pseudoconvex
  boundary points of $b\Omega$ are of strict type $4$. Then $\Omega$ admits a smooth defining function which is
  plurisubharmonic on $b\Omega$.
\end{corollary}

\begin{proof}
  Let $r \colon V \to \R$ be a smooth defining function for $\Omega$.
  Let $\mathcal{W} \subset b\Omega$ denote the set of points at which $\Omega$ is weakly pseudoconvex. 
  Then $\mathcal{W}$ is closed in $b\Omega$. Further, it follows from the hypothesis that 
  $B = |L\bar{L}\lambda|^2-|LL\lambda|^2$ 
  is strictly positive on some open neighborhood $U \Subset V$ of $\mathcal{W}$. 
  As in the proof of Theorem \ref{T:type4pshonboundary}, we find a smooth function 
  $h \colon U \to \R$ such that $\rho \coloneqq re^h$ satisfies \eqref{E:pshonboundary} on $b\Omega\cap U$. 
  
  Let $U'\Subset U$ be another open neighborhood of $\mathcal{W}$, and let $\chi$ be a real-valued, smooth function 
  which is compactly supported in $U$ and identically $1$ on $U'$. Then $\tilde\rho:=r e^{\chi\cdot h}$ is a smooth
  defining function for $\Omega$ such that $\tilde\rho = \rho$ on $U'$, hence it satisfies \eqref{E:pshonboundary}
  on $b\Omega\cap U'$. Now note that $b\Omega\setminus U'$ is a compact set at whose points $\Omega$ is strictly
  pseudoconvex. Hence, \eqref{E:pshonboundary} is satisfied on $b\Omega\setminus U'$ by any defining 
  function for $\Omega$, in particular by $\tilde\rho$. Thus, $\tilde\rho$ satisfies \eqref{E:pshonboundary}
  on all of $b\Omega$. It now follows from \Cref{L:pshonboundaryiff} that $\Omega$ admits a defining function which is plurisubharmonic on $b\Omega$.
\end{proof}

In the following, we give two examples of smoothly bounded, pseudoconvex domains which admit a plurisubharmonic defining function  on the boundary, although they have boundary points of weak type $4$.

\begin{example}
  For $(z,w)\in\mathbb{C}^2$, write $z=x+iy$ and $w=u+iv$. 
  Then define $\Omega=\{(z,w)\in\mathbb{C}^2: r(z,w)<0\}$
  with $r(z,w):=u+f(z)$ for some smooth, subharmonic function $f$.
   It follows that for all $\in\Vhol{}{\mathbb{C}^2}$, $V = V^1\frac{\partial}{\partial z} + V^2 \frac{\partial}{\partial w}$,
  $$\Hc{r}{V}{V}=f_{z\bar{z}}|V^1|^2. $$
  %\;\;\sjump\forall\;V = V^1\textstyle\frac{\partial}{\partial z} + V^2 \frac{\partial}{\partial w}\in\Vhol{}{\mathbb{C}^2}.
  Hence, $r$ is a plurisubharmonic defining function for $\Omega$, independent of the type of $b\Omega$ at any of its boundary points.
  In particular, $b\Omega$ may be of weak type $4$ at some boundary point $p_0$, e.g., if $f(z)=x^4$ and $p_0=0$.
\end{example}

\begin{example}\label{E:example}
  As in the previous example, write $z=x+iy$ and $w=u+iv$. Set $$U=\{(z,w)\in\mathbb{C}^2:|x|<\pi/2\},$$
  and define $\Omega=\{(z,w)\in U:r(z,w)<0\}$ for
  $$r(z,w)=u-\textstyle\frac{1}{2}\left(x-v \right)^2-\ln(\cos(x)).$$
  %Again, we set $L=r_w\frac{\partial}{\partial z}-r_z\frac{\partial}{\partial w}$. To show that $\Omega$ is pseudoconvex, 
  %it needs to be shown that
  %$$\Hc{r}{L}{L}(p)=\left(r_{z\bar{z}}|r_w|^2-2\re\left(r_{z\bar{w}}r_w r_{\bar{z}}\right)+r_{w\bar{w}}|r_z|^2\right)(p)\geq 0\;\;\sjump\forall\; p\in b\Omega.$$
  We compute that, for every $(z,w)\in U$,
  \begin{align*}
    r_z(z,w)&=\textstyle \frac{1}{2}\left(\tan(x)-\left(x-v\right)\right),\\
    r_w(z,w)&=\textstyle \frac{1}{2}(1-i(x-v))\\
    r_{z\bar{z}}(z,w)&=\textstyle\frac{1}{4}\tan^2(x),\;r_{z\bar{w}}(z,w)=\textstyle\frac{i}{4},\;r_{w\bar{w}}(z,w)=-\textstyle\frac{1}{4},
   \end{align*}
   so that
   \begin{align*}
     (r_{z\bar{z}}|r_w|^2)(z,w)&=\textstyle{\frac{1}{16}}\tan^2(x)(1+(x-v)^2),\\
     (r_{w\bar{w}}|r_z|^2)(z,w)&=-\textstyle{\frac{1}{16}}\left(\tan(x)-(x-v)\right)^2,\\
     -2\re[r_{z\bar{w}}r_wr_{\bar{z}}](z,w)&=-\textstyle{\frac{1}{8}}(x-v)\tan(x)
     +\textstyle{\frac{1}{8}}\left(x-v\right)^2.
   \end{align*}
   Hence, for $L=r_w\frac{\partial}{\partial z}-r_z\frac{\partial}{\partial w}$, we obtain
   \begin{align*}
     \Hc{r}{L}{L}(z,w)=\textstyle{\frac{1}{16}}(x-v)^2\sec^2(x)\geq 0,
   \end{align*}
   that is, $\Omega$ is pseudoconvex. In fact, $\Omega$ is strictly pseudoconvex except at boundary points 
   satisfying $x=v$. 
   Moreover, since $-\ln(\cos(x)) = \frac{1}{2}x^2 + \frac{1}{12}x^4 + o(x^4)$, it follows from
   \eqref{E:FiniteTypeCoordinates} and \Cref{thm:stricttype4incoordinates} that $\Omega$ is 
   of weak type $4$ at the origin.
   In particular, the function $h$ constructed in the proof of 
   Theorem~\ref{T:type4pshonboundary}, see \eqref{E:definitionh}, is not defined at the origin.
   However, a straightforward computation, see \Cref{E:examplecont} in Section \ref{S:sesquiconvex}, shows that
   $\rho(z,w):=r(z,w)e^{y}\cos x$ satisfies \eqref{E:pshonboundary}, i.e., for every $V \Subset U$ there exists 
   a function $\hat{\rho} \colon U \to \R$ such that $\hat{\rho}$ is plurisubharmonic on $b\Omega \cap V$, 
   see \Cref{R:pshonboundaryiff}.
   
   We show in the following that any smooth function $h$ such that $re^h$ is plurisubharmonic on   
   $b\Omega$ near the origin must have nonvanishing derivative with respect to $y$ at the origin, although $r$ is independent of $y$. This is  
   noteworthy 
   because it shows that, in the case of the domain being of weak type $4$ at some boundary point, 
   the multiplier function $h$, if it exists, can in general not be given as a combination of derivatives of 
   $r$ as in the strict type $4$ 
   case.
   
   Now, suppose $h$ is a positive, smooth function near the origin such that $h_y(0)=0$. A straightforward computation, with $V=-\frac{\partial}{\partial z}+is\frac{\partial}{\partial w}$ for $s>0$, yields
   \begin{align*}
     \Hc{r}{V}{V}(0)&=
     %-\textstyle{\frac{1}{4}}|V^2|^2+2\re(\textstyle{\frac{i}{4}}V^1\bar{V}^2)=
     -\textstyle\frac{1}{2}s+\mathcal{O}(s^2).
   \end{align*}
   Moreover,
   \begin{align*}
     (Vr)(0)&=%r_w(0)V^2=
     \textstyle\frac{i}{2}s, \text{ and}\\
     (Vh)(0)&=%h_z(0)V^1+h_w(0)V^2=
     -\textstyle\frac{1}{2}h_x(0)+ish_w(0).
   \end{align*}
  Therefore,
   \begin{align*}  
     2\re\left(Vh\cdot\bar{V}r \right)(0)
     %=-\re\Bigl(\left(-\textstyle{\frac{1}{2}}h_x(0)+ish_w(0)\right)\cdot is \Bigr)
     =\mathcal{O}(s^2).
   \end{align*}
   It then follows that
   \begin{align*}
     \Hc{e^h r}{V}{V}(0)=e^{h(0)}\Bigl( \Hc{r}{V}{V}
     +2\re\left(Vh\cdot\bar{V}r \right) \Bigr)(0)
     =-\textstyle{\frac{1}{2}}s+\mathcal{O}(s^2)<0
   \end{align*} 
   for all $s>0$ sufficiently close to $0$.
\end{example}

\section{Plurisubharmonicity near the boundary}\label{sec:pshnearboundary}

In this section, we first consider smoothly bounded, pseudoconvex domains in $\mathbb{C}^2$ such that all weakly pseudoconvex boundary points are of type $4$. 
In the case of bounded domains, we show that there exists a smooth, plurisubharmonic defining function for the domain whenever there is a smooth defining function which is plurisubharmonic on the boundary of the domain.

In the latter part of this section, we consider smoothly bounded, pseudoconvex domains in $\mathbb{C}^2$ that are at least of type 6 at their weakly pseudoconvex boundary points. In the case that such a domain admits a smooth defining function which is plurisubharmonic on the boundary, we give a simplified proof that both the Diederich--Forn{\ae}ss index and the Steinness index are $1$.

%We start with proving a sufficient condition for the existence of smooth, plurisubharmonic local defining functions.
A lack of understanding of the notion of existence of a plurisubharmonic defining function is rooted in the lack of an equivalent condition which is checkable for \emph{any} defining function. The following lemma yields a condition which is checkable on a class of defining functions strictly larger than the class of plurisubharmonic defining functions.
A version of this lemma in the context of convex domains is given in \cite[Proposition 6.17]{HerMcN09}.

\begin{lemma}\label{L:basicestimate1}
  Let $\Omega\subset\mathbb{C}^n$ be a smoothly bounded, pseudoconvex domain, $p_0\in b\Omega$.
  %Then $\Omega$ admits a smooth local defining function which is plurisubharmonic near $p_0$ 
  %if and only if there exist an open neighborhood $U$ of 
  %$p_0$ and a smooth local defining function $r$ of $\Omega$ on $U$ such that for some constant $C>0$
  Then $\Omega$ admits a smooth local defining function which is plurisubharmonic near $p_0$ 
  if and only if there exists a smooth local defining function $r$ for $\Omega$ near $p_0$ such that 
  \begin{align}\label{E:basicestimate1}
    \Hc{r}{\xi}{\xi}\geq -C\left( r^2|\xi|^2+\left|\langle\partial r,\xi\rangle\right|^2\right)\;\;\sjump\forall\;\xi\in\mathbb{C}^n
  \end{align}
  for some constant $C>0$.\footnote{Here, and occasionally later on, we consider the complex Hessian form $\Hcsymbol{r}$ at a point as a sesquilinear form on $\mathbb{C}^n$.}
\end{lemma}

%In this section, we at times consider the complex Hessian matrix of a function at a point a sesquilinear from on $\mathbb{C}^n$.
 To put \eqref{E:basicestimate1} in context, we recall that for any smoothly bounded, pseudoconvex domain $\Omega\Subset\mathbb{C}^n$ with smooth defining function $r$, 
there exist an open neighborhood $U$ of the boundary of $\Omega$ and a constant $C>0$ such  that
\begin{align}\label{E:RangeEst}
  \Hc{r}{\xi}{\xi}\geq -C\left( |r|\cdot|\xi|^2+\left|\langle\partial r,\xi\rangle\right|\cdot|\xi|\right)\;\;\sjump\forall\xi\in\mathbb{C}^n
\end{align}
on $U$.
That \eqref{E:RangeEst} holds true on $\Omega\cap U$ is derived by Range, see (5) in \cite{Range81}, to reprove the result of Diederich--Forn{\ae}ss \cite[Theorem 1]{DieFor77-2} on the existence of bounded, strictly 
plurisubharmonic exhaustion functions for smoothly bounded, pseudoconvex domains, see \cite[Theorem 2]{Range81}. Arguments similar to the ones in \cite{Range81} yield \eqref{E:RangeEst} on $U$. We note that, if for every $\varepsilon>0$, there exists a smooth defining function $r=r_\varepsilon$ such that \eqref{E:RangeEst} holds with $C=\varepsilon$, then both the Diederich--Forn{\ae}ss index and the Steinness index are $1$, see the proof of Corollary 1.6 in \cite{ForHer08}.

\begin{proof}[Proof of Lemma \ref{L:basicestimate1}]
  Note first that if $r$ is a smooth local defining function for $\Omega$ which is plurisubharmonic on an open neighborhood $U$ of $p_0$, then 
   \eqref{E:basicestimate1} holds trivially for $r$ on $U$. 
   
  Let $p_0\in b\Omega$ and let $U \Subset \C^n$ be an open neighborhood of $p_0$. 
  Now suppose that $r: U\longrightarrow \mathbb{R}$ is a smooth local defining function for $\Omega$ on $U$ such that  \eqref{E:basicestimate1} holds.
  Consider $\rho:=r+r^2\psi$ with $\psi(z):=K_1+K_2|z|^2$ for fixed, positive constants $K_1$ and $K_2$ to be determined later. It follows from a straightforward computation that
  \begin{align*}
    \Hc{\rho}{\xi}{\xi}=(1+2r\psi)\Hc{r}{\xi}{\xi}+2\psi|\langle\partial r,\xi\rangle|^2+4K_2r\re\left(\langle\partial r,\xi\rangle \langle z,\xi\rangle\right)+r^2 K_2|\xi|^2
  \end{align*}
  for $\xi\in\mathbb{C}^n$. Next, it follows from 
  \begin{align}\label{E:sclc}
     2ab\leq\varepsilon a^2+\frac{1}{\varepsilon}b^2 \;\;\text{ for }a,b\geq 0 \text{ and }\varepsilon>0,
  \end{align}
  with $a = \abs{r}\abs{\xi}$, $b = \abs{z}\abs{\langle\partial r,\xi\rangle}$ and $\varepsilon = \frac{1}{4}$, that
  \begin{align*}
    4K_2\left|r\re\bigl(\langle\partial r,\xi\rangle\langle z,\xi\rangle\bigr) \right|\leq r^2\frac{K_2}{2}|\xi|^2+8K_2|z|^2|\langle\partial r,\xi\rangle|^2
  \end{align*}
  holds.
  Therefore, we obtain
  \begin{align*}
      \Hc{\rho}{\xi}{\xi}\geq(1+2r\psi)\Hc{r}{\xi}{\xi}+2\left|\langle\partial r,\xi\rangle \right|^2\left(K_1-3K_2|z|^2 \right)+r^2\frac{K_2}{2}|\xi|^2.
  \end{align*}
  Next, for each $K:=(K_1,K_2)$, there exists an open neighborhood $U_{K}\subset U$ of $p_0$ such that 
  $$1+2r(z)(K_1+K_2|z|^2)\leq 3/2$$
  holds for all  $z\in U_{K}$. Note that $U_K$ may be chosen such that $b\Omega\cap U=b\Omega\cap U_K$.
  Using (\ref{E:basicestimate1}), we then obtain on $U_K$
  \begin{align*}
    (1+2r\psi)\Hc{r}{\xi}{\xi} %&\geq -(1+2r\psi)C\left(r^2|\xi|^2+|\langle\partial r,\xi\rangle|^2\right)\\
    &\geq-\frac{3C}{2}\left(r^2|\xi|^2+|\langle\partial r,\xi\rangle|^2\right)
  \end{align*}
  for all $\xi\in\mathbb{C}^n$.
  Therefore,
  \begin{align*}
    \Hc{\rho}{\xi}{\xi}
     \geq 
    r^2 |\xi|^2\left(\frac{K_2}{2}-\frac{3}{2}C \right)+2\left|\langle\partial r,\xi\rangle\right|^2\left(K_1-3K_2|z|^2-\frac{3}{4}C\right)
  \end{align*}
  holds on $U_K$ for all $\xi\in\mathbb{C}^n$.
  Fix $K_2$ such that $K_2>3C$ holds. Let $D$ be the maximum of $|z|$ on $U$, then fix $K_1$ such that $K_1> 3K_2D^2+\frac{3}{4}C$ holds. 
  %It follows that for any such pair $K=(K_1,K_2)$, the local defining function $\rho$ is plurisubharmonic on $U_K$. In particular, there exists a positive constant $c$ such that
  It follows easily that there exists a positive contant $c>0$ such that
  \begin{align} \label{E:HessianPositivity}
     \Hc{\rho}{\xi}{\xi}\geq c\left(\rho^2|\xi|^2+|\langle\partial\rho,\xi\rangle|^2\right)
  \end{align}
  on $U_K$ for all $\xi\in\mathbb{C}^n$, i.e., $\rho$ is plurisubharmonic on $U_K$.
\end{proof}
\begin{remark}
  %Note that $\rho$ is strictly plurisubharmonic on $U_K\setminus b\Omega$.
  %Moreover, the complex Hessian of $\rho$ is positive definite at strictly pseudoconvex boundary points of $\Omega$. 
  Note that the function $\rho$ constructed in the proof of \Cref{L:basicestimate1} is strictly plurisubharmonic on $U_K\setminus b\Omega$, see (\ref{E:HessianPositivity}).
  Moreover, the complex Hessian of $\rho$ is positive definite at strictly pseudoconvex boundary points of $\Omega$.  
  %Moreover, the complex Hessian of $\rho$ is positive definite at all points $p \in b\Omega \cap U_K$ at which $\Omega$ is strictly pseudoconvex.
  To wit, the complex Hessian of $\rho$ is strictly positive in non-zero complex tangential 
  directions at these boundary points by definition, and it is strictly positive in all directions with a 
  non-vanishing normal component to the boundary by (\ref{E:HessianPositivity}).
  %It follows that the complex Hessian of $\rho$ is positive definite at all points of $U_K$ except at those boundary points at which $\Omega$ is weakly pseudoconvex.
\end{remark}

We note that a global version of Lemma \ref{L:basicestimate1} holds if $\Omega$ is bounded and $U$ is an open neighborhood of $b\Omega$. Moreover, by a result of Morrow--Rossi \cite[Lemma 1.3]{Morrow-Rossi75}, see also \cite{Morrow-Rossi77}, any smoothly bounded, strictly pseudoconvex, bounded domain in $\mathbb{C}^n$ admits a smooth defining function which is strictly plurisubharmonic in an open neighborhood of the closure of the domain. The same argument as the one used in the proof of Lemma 1.3 in \cite{Morrow-Rossi75} yields the following.

\begin{corollary}\label{C:psheverywhere}
  Let $\Omega\Subset\mathbb{C}^n$ be a smoothly bounded domain.
  Assume that there exists a smooth defining function $r$ for $\Omega$ such that
   \begin{align*}
     \Hc{r}{\xi}{\xi}\geq -C\left( r^2|\xi|^2+\left|\left\langle\partial r,\xi\right\rangle \right|^2\right)\;\;\forall\xi\in\mathbb{C}^n
   \end{align*}
   holds near $b\Omega$ for some constant $C>0$. Then there exists a smooth defining function $\rho$ for $\Omega$ on an open neighborhood of $\bar{\Omega}$ such that
  \begin{align*}
     \Hc{\rho}{\xi}{\xi}\geq c\left( \rho^2|\xi|^2+\left|\left\langle\partial \rho,\xi\right\rangle \right|^2\right)\;\;\forall\xi\in\mathbb{C}^n
   \end{align*}
   holds for some $c>0$.
 \end{corollary}

Whether a given local defining function actually satisfies condition \eqref{E:basicestimate1} in some  open neighborhood of the boundary, can be detected from the behaviour of the complex Hessian of that defining function and its normal derivative  \emph{on the boundary of the domain} as follows.

\begin{proposition}\label{L:pshnearboundaryiff}
  Let $\Omega\subset\mathbb{C}^2$ be a smoothly bounded, pseudoconvex domain, $p_0\in b\Omega$. Then $\Omega$ admits a smooth local defining function which is plurisubharmonic near $p_0$ 
  if and only if there exist
  an open neighborhood $U$ of $p_0$ and a smooth local defining function $r$ for $\Omega$ on $U$ such that
  \begin{align}
    \left|\Hc{r}{L}{N}\right|^2&=\mathcal{O}\left(\Hc{r}{L}{L}\right)\;\;\text{on}\;\;b\Omega\cap U,\;\;\text{and}\label{E:pshnearboundaryiff1}\\
    \left|\nu \Hc{r}{L}{L} \right|^2&=\mathcal{O}\left(\Hc{r}{L}{L}\right)\;\;\text{on}\;\;b\Omega\cap U.\label{E:pshnearboundaryiff2}
  \end{align}
\end{proposition}

\begin{proof}
    Let $r \colon U \to \R$ be a smooth local defining function for $\Omega$ near $p_0$, and let $L$ be a nonvanishing
    holomorphic tangential vector field on $U$. Note that the normal derivative $(\nu \Hc{r}{L}{L})_{|_{b\Omega \cap U}}$    
    depends on $L$, while $\Hc{r}{L}{L}_{|_{b\Omega \cap U}}$ depends only on $L_{|_{b\Omega \cap U}}$.
    However, in case that (\ref{E:pshnearboundaryiff1}) holds true, the condition
    (\ref{E:pshnearboundaryiff2}) is in fact independent of the choice of $L$. To see this, let $L'$ be another
    nonvanishing holomorphic tangential vector field on $U$. Then there exist a smooth function $h \colon U \to \C$ 
    and a vector field $E \in \Vhol{}{U}$ such that $L' = hL + rE$. Thus
    \begin{align*}
      \Hc{r}{L'}{L'}=\abs{h}^2\Hc{r}{L}{L}
      +2r\re\bigl(h\Hc{r}{L}{E}\bigr)+r^2\Hc{r}{E}{E},
    \end{align*}
    so that
    \begin{align*}
       \nu\Hc{r}{L'}{L'}
       =\abs{h}^2\nu\Hc{r}{L}{L} + 2\abs{dr}\re\bigl(h\Hc{r}{L}{E}\bigr) + \mathcal{O}(\lambda) \;\;\text{on } 
       b\Omega\cap U.
    \end{align*}
    Since $E=aL + bN$ on $b\Omega \cap U$ for smooth functions $a,b$, it follows from  
    \eqref{E:pshnearboundaryiff1} that $\Hc{r}{L}{E} = \mathcal{O}(\sqrt{\lambda})$ on $b\Omega \cap U$. In particular,
    $\nu\Hc{r}{L'}{L'} = \abs{h}^2\nu\Hc{r}{L}{L} + \mathcal{O}(\sqrt{\lambda}) \;\;\text{on } b\Omega\cap U$,   
    which shows that (\ref{E:pshnearboundaryiff2}) is well-defined.

    Now suppose first that $\Omega$ admits a plurisubharmonic, smooth local defining function $r \colon U \to \R$ 
    near $p_0$. Then \eqref{E:pshnearboundaryiff1} 
    holds by \Cref{L:pshonboundaryiff}.
    Moreover, an application of the first part of \Cref{thm:lHospital},
    with $f=\Hc{r}{L}{L}$, shows that \eqref{E:pshnearboundaryiff2} is satisfied.
    
    On the other hand, let $r \colon U \to \R$ be a smooth local defining function for $\Omega$ near $p_0$ such that
    (\ref{E:pshnearboundaryiff1}) and (\ref{E:pshnearboundaryiff2}) hold true. After possibly shrinking $U$
    in the direction normal to $b\Omega$, 
    we may assume that there exists a smooth map $\pi \colon U \to b\Omega$ 
    such $\abs{q-\pi(q)}=d_{b\Omega}(q)$.
    Fix $q\in U$, and set $p \coloneqq \pi(q)$. Moreover, fix $\xi\in\mathbb{C}^2$, and write $\xi=aL(q)+bN(q)$ 
    for some $a = a(q,\xi), b=b(q,\xi)\in\mathbb{C}$. Then
    \begin{align*}
      \Hc{r}{\xi}{\xi}(q)=|a|^2\Hc{r}{L}{L}(q)+2\re&(a\bar{b}\Hc{r}{L}{N}(q))+|b|^2\Hc{r}{N}{N}(q).
    \end{align*}
    In view of \eqref{E:Taylor}, with $f = \Hc{r}{L}{L}$, the Taylor expansion at $p$ in direction $\nu$ gives
    \begin{align*}  
     \Hc{r}{\xi}{\xi}(q)&=|a|^2\left(\Hc{r}{L}{L}(p)+\delta_{b\Omega}(q)(\nu\Hc{r}{L}{L})(p)
     +\mathcal{O}(d_{b\Omega}^2)(q)\right)\\
     &+2\re\left(a\bar{b}(\Hc{r}{L}{N}(p) +\mathcal{O}(d_{b\Omega})(q))\right)\\
     &+\abs{b}^2\mathcal{O}(1)(q),
    \end{align*}
    where the $\mathcal{O}$-terms are functions that do not depend on $q$ and $\xi$. 
    Fix an open neighborhood $V \Subset U$ of $p_0$.
    Using \eqref{E:sclc}, it follows from \eqref{E:pshnearboundaryiff2} 
    and \eqref{E:pshnearboundaryiff1} that there exist constants $C_1,C_2 > 0$ such that on $V$
    \begin{align*}
     \textstyle\frac{1}{2} \Hc{r}{L}{L}(p) + \delta_{b\Omega}(q)(\nu\Hc{r}{L}{L})(p) 
     &\ge -C_1d_{b\Omega}^2(q),\\
     \textstyle\frac{1}{2} \abs{a}^2 \Hc{r}{L}{L}(p) + 2\re\left(a\bar{b}\Hc{r}{L}{N}(p)\right) 
     &\ge - C_2\abs{b}^2.
    \end{align*}  
    Hence, there exists a constant $C_3>0$, which does not depend on $q$ and $\xi$, such that
    \begin{align*}
     \Hc{r}{\xi}{\xi}(q) \ge -C_3(|a|^2d_{b\Omega}^2(q) + \abs{b}^2)\;\;\sjump\forall\; \xi\in\mathbb{C}^n\;\;
       \forall\;q\in V.
    \end{align*}  
    Since $|a|\leq|\xi|$, $|b| = \mathcal{O}(|\langle\partial r,\xi\rangle|)$ on $U$, and 
    $d_{b\Omega}^2 = \mathcal{O}(r^2)$ on $U$, it follows that $r_{|_{V}}$ satisfies \eqref{E:basicestimate1}. 
    The claim thus follows from \Cref{L:basicestimate1}.
    %It now follows from Lemma \ref{L:basicestimate1} and its proof that $\Omega$ has a smooth 
    %local defining function which is plurisubharmonic in some open neighborhood $V$ of $p_0$ which
    %satisfies $V\cap b\Omega =U\cap b\Omega$.
\end{proof}

\begin{remark} \label{R:pshnearboundaryiff}
The proofs of \Cref{L:basicestimate1} and \Cref{L:pshnearboundaryiff} imply the following:  
if \eqref{E:pshnearboundaryiff1} and \eqref{E:pshnearboundaryiff2} hold true, then 
for every $K \Subset b\Omega \cap U$ there exist an open neighborhood $V \subset U$ of $K$ and a smooth local defining function for $\Omega$ on $U$ that is plurisubharmonic on $V$.
\end{remark}

\begin{remark}
Versions of \Cref{L:pshonboundaryiff} and \Cref{L:pshnearboundaryiff} can also be shown for domains in $\mathbb{C}^n$, $n > 2$. In this case, $L$ has to be substituted by a frame $\{L_j\}_{j=1}^{n-1}$ for $T(b\Omega)^{1,0}$ near $p_0$.
\end{remark}

Condition \eqref{E:pshnearboundaryiff2} may always be achieved near boundary 
points of type $4$, independent of whether the smoothly bounded, pseudoconvex 
domain in consideration actually admits a smooth local defining function which is plurisubharmonic on the boundary.

\begin{lemma}\label{L:pshnearboundaryiff2}
    Let $\Omega\subset\mathbb{C}^2$ be a smoothly bounded, pseudoconvex domain, $p_0 \in b\Omega$. 
    If $b\Omega$ is of type $4$ at $p_0$, then for every nonvanishing holomorphic tangential vector field $L$ near $p_0$ 
    there exists a smooth local defining function $\rho \colon U \to \R$ for $\Omega$ near $p_0$ such that
    \begin{align}\label{E:pshnearboundaryiff2rho}
       \left|\nu \Hc{\rho}{L}{L} \right|^2&=\mathcal{O}\left(\Hc{\rho}{L}{L}\right)\;\;\text{on}\;\;b\Omega\cap U.
    \end{align}
\end{lemma}

\begin{proof}
  Let $r \colon U \to \R$ be a smooth local defining function for $\Omega$ near $p_0$, and let $L$ be a 
  nonvanishing holomorphic tangential vector field near $p_0$.
  After possibly shrinking $U$, we may assume that $dr \neq 0$, and that $L$ is defined on $U$.
  Set $\Lambda \coloneqq \Hc{r}{L}{L}$ and 
  $\lambda \coloneqq \Lambda_{|_{b\Omega\cap U}}$.
  We claim that, after possibly shrinking $U$, for every smooth function
  $F \colon U \to \mathbb{R}$, there exists a smooth function $h \colon U\to\mathbb{R}$ such that 
  \begin{align}\label{E:2ndPDE}
  \begin{cases}
      &Lh=\mathcal{O}(\sqrt{\lambda})\\
      &L\bar{L}h=F+\mathcal{O}(\sqrt{\lambda})
      %&\Hc{h}{L}{L}=F+\mathcal{O}(\sqrt{\lambda})    
  \end{cases}
  \;\;\text{ on }b\Omega\cap U.
  \end{align}
  Indeed, define smooth functions $A,B \colon U \to \R$ by 
  \begin{align*}
   A \coloneqq |L\Lambda|^2 \quad\text{and}\quad
   B \coloneqq \Lambda^2+\abs{L\bar{L}\Lambda}^2+\abs{LL\Lambda}^2,
  \end{align*}
  and set
  \begin{align}\label{E:mdefinition}
  h \coloneqq F \frac{A}{B}.
  \end{align}
  Since $\Omega$ is of type 4 at $p_0$, after possibly shrinking $U$, we can assume that $B>0$ on $U$. 
  Thus, $h$ is well-defined. Moreover, since $L\Lambda = \mathcal{O}(\sqrt{\lambda})$ on $b\Omega \cap U$ by
  \eqref{equ:VectorFieldlHospital}, 
  %it follows that $A$ and $\abs{dA}$ are of class $\mathcal{O}(\sqrt{\lambda})$ on $b\Omega \cap U$. Hence, 
  it follows that $A$ and $LA$ are of class $\mathcal{O}(\sqrt{\lambda})$ on $b\Omega \cap U$, and, in particular, that
  \begin{align} \label{E:mestimate}
  %\abs{dh} = \mathcal{O}(\sqrt{\lambda}) \quad\text{on } b\Omega \cap U,
  Lh = \mathcal{O}(\sqrt{\lambda}) \quad\text{on } b\Omega \cap U.
  \end{align}
  Moreover,
  $$
  L\bar{L}h = F \frac{L\bar{L}A}{B} + \mathcal{O}(\sqrt{\lambda}) = F\frac{\abs{L\bar{L}\Lambda}^2+\abs{LL\Lambda}^2}{B} + \mathcal{O}(\sqrt{\lambda}) = F + \mathcal{O}(\sqrt{\lambda}) \;\;\text{ on }b\Omega\cap U.
  $$
  
  Now let $h \colon U \to \R$ be an arbitrary smooth function, and set $\rho \coloneqq re^h$. Then
  \begin{align*}
    \Hc{\rho}{L}{L} = e^h\left(\Hc{r}{L}{L}+ r\left(\Hc{h}{L}{L} + \abs{Lh}^2\right) \right),
  \end{align*}
  and thus
  \begin{align*}
    \nu\Hc{\rho}{L}{L} = e^h\left(\nu\Hc{r}{L}{L} + \abs{dr}\left(\Hc{h}{L}{L} + \abs{Lh}^2\right) \right) + \mathcal{O}(\lambda) \quad\text{on } b\Omega \cap U.
  \end{align*} 
  %Since $\Hc{h}{L}{L} = L\bar{L}h - (\nabla_L\bar{L})h$, it follows from \eqref{E:mestimate} that
  Since $\Hc{h}{L}{L} = L\bar{L}h - (\nabla_L\bar{L})h$, and since on $b\Omega \cap U$ it follows from $\Hc{r}{L}{L} = -(\nabla_L\bar{L})r$ 
  on $b\Omega \cap U$ that the component of $\nabla_L\bar{L}$ normal to $b\Omega$ is of the form 
  $\mathcal{O}(\lambda)$, it follows from \eqref{E:mestimate} that
  \begin{align*}
    \nu\Hc{\rho}{L}{L} = e^h\left(\nu\Hc{r}{L}{L} + \abs{dr}L\bar{L}h  \right) + \mathcal{O}(\sqrt{\lambda}) \quad\text{on } b\Omega \cap U.
  \end{align*} 
  Thus, if $h$ is the solution for \eqref{E:2ndPDE} with 
  \begin{align} \label{E:FData}
  F \coloneqq -\frac{1}{\abs{dr}}\nu\Hc{r}{L}{L},
  \end{align}
  then $\rho$ satisfies \eqref{E:pshnearboundaryiff2rho}.  
\end{proof}

We now can prove the main result of this section.

\begin{theorem}\label{T:pshnearboundary}
  Let $\Omega\subset\mathbb{C}^2$ be a smoothly bounded, pseudoconvex domain.
  Suppose $p_0\in b\Omega$ is such that $c_{p_0}=4$. If $\Omega$ admits a 
  smooth local defining function near $p_0$ 
  which is plurisubharmonic on $b\Omega$ near $p_0$,
  then $\Omega$ admits a smooth local defining function near $p_0$ which is plurisubharmonic.
\end{theorem}

\begin{proof}
  %Let $U$ be an open neighborhood of $p_0$ such that $c_p\leq 4$ for all $p\in b\Omega\cap U$. 
  Let $r \colon U \to \R$ be a smooth local defining function for $\Omega$ near $p_0$ such that 
  $r$ is plurisubharmonic on $b\Omega\cap U$. 
  After possibly shrinking $U$, we can assume that $c_p \le 4$ for all $p \in b\Omega \cap U$.
  Set $\rho \coloneqq re^h$, where $h$ is the solution to 
  \eqref{E:2ndPDE} with $F$ given as in \eqref{E:FData}. Then, as in the proof of
  \Cref{L:pshnearboundaryiff2}, we see that $\rho$ satisfies \eqref{E:pshnearboundaryiff2}. 
  On the other hand, since, by \Cref{L:pshonboundaryiff}, $r$ satisfies
  \eqref{E:pshnearboundaryiff1}, and since 
  $$
  \Hc{\rho}{L}{N} = e^h\left(\Hc{r}{L}{N} + Lh \cdot \bar{N}r \right)\;\;\text{on } b\Omega\cap U,
  $$ 
  it follows with \eqref{E:mestimate} that $\rho$ satisfies \eqref{E:pshnearboundaryiff1}. 
  The claim thus follows from \Cref{L:pshnearboundaryiff}.  
\end{proof}

 Note that, if $r$ is a smooth defining function for a smoothly bounded, pseudoconvex domain $\Omega \subset \C^2$ such that $\Omega$ is of type $4$ at all its weakly pseudoconvex boundary points, then the function $h$ defined in \eqref{E:mdefinition} and \eqref{E:FData} is defined in an open neighborhood of $b\Omega$. Moreover, the function $h$ solves \eqref{E:2ndPDE} on $b\Omega$. In view of \Cref{R:pshnearboundaryiff}, this implies the following global result.
\begin{corollary}
  Let $\Omega\Subset\mathbb{C}^2$ be a smoothly bounded, pseudoconvex domain. Suppose that $\Omega$ has a smooth defining function which is plurisubharmonic on $b\Omega$, and that $c_p\leq 4$ for all
  $p\in b\Omega$. Then $\Omega$ admits a smooth defining function which is plurisubharmonic in an open neighborhood of $b\Omega$.
\end{corollary}

An analogon to Theorem \ref{T:pshnearboundary} near higher order boundary points is not apparent, 
although condition \eqref{E:pshnearboundaryiff2} always holds \emph{at} boundary points 
of type larger than $4$, whenever $r$ is a defining function that is plurisubharmonic on the boundary, as shown by the next Lemma.

\begin{lemma}\label{L:type6result}
   Let $\Omega\subset\mathbb{C}^2$ be a smoothly bounded pseudoconvex domain. Let $p_0\in b\Omega$ such that $c_{p_0}\geq 6$. If $r$ is a smooth local
   defining function of $b\Omega$ near $p_0$ which is plurisubharmonic on $b\Omega$ near $p_0$, 
   and if $L$ is a nonvanishing holomorphic tangential vector field near $p_0$, then
   \begin{align*}
    \nu \Hc{r}{L}{L}(p_0)=0.
   \end{align*}
\end{lemma}   

\begin{proof}
   Let $r$ be a smooth local defining function for $\Omega$ on some open neighborhood $U$ of $p_0$ such that $r$ is
   plurisubharmonic on $b\Omega\cap U$. After possibly shrinking $U$, we may assume that $N \coloneqq N_r$ is 
   defined on $U$, see \eqref{E:definitionN}. As usual, we write $L=\frac{1}{2}(X+iY)$ with $X,Y \in \Vr{}{U}$. 
   
   Consider the function $g \colon b\Omega \cap U \to \R$ given by
   $$
   g = \Hc{r}{L}{L}\Hc{r}{N}{N}-|\Hc{r}{L}{N}|^2 \quad\text{ on } b\Omega \cap U.
   $$
   Since $r$ is plurisubharmonic on $b\Omega \cap U$, it follows that $g \ge 0$, and since $\Omega$ 
   is weakly pseudoconvex at $p_0$, it follows that $g(p_0) = 0$. Thus, $(XXg)(p_0) \ge 0$ and $(YYg)(p_0) \ge 0$, see \Cref{L:XXf}. Since $L\bar{L}g = \frac{1}{4}(XXg+YYg) - \frac{i}{4}[X,Y]g$, and since the tangential derivative $[X,Y]g$ vanishes at $p_0$, it follows that $(L\bar{L}g)(p_0) \ge 0$. Moreover, the fact that $c_{p_0} > 4$ implies that for $\lambda \coloneqq \Hc{r}{L}{L}_{|_{b\Omega \cap U}}$ the functions $\lambda, L\lambda, \bar{L}\lambda, L\bar{L}\lambda$ all vanish at $p_0$. Hence, since $\Hc{r}{L}{N}(p_0) = 0$,
   $$
    (L\bar{L}g)(p_0) = -\left|L\Hc{r}{L}{N}\right|^2(p_0) -\left|\bar{L}\Hc{r}{L}{N}\right|^2(p_0).
   $$
   In particular, it follows that $\bar{L}\Hc{r}{L}{N}(p_0) = 0$.
   
   However, if we denote coordinates in $\C^2$ by $z = (z_1, z_2)$, and if we write $L = \sum_{j=1}^2L^j\frac{\partial}{\partial z_j}$, $N = \sum_{j=1}^2N^j\frac{\partial}{\partial z_j}$, then
   \begin{align}
     \label{E:SwitchDerivates1}
     \bar{L}\Hc{r}{L}{N} &=
     \sum_{j,k,\ell=1}^2\frac{\partial^3 r}{\partial\bar{z}_{\ell}\partial z_j\partial\bar{z}_k}
     \bar{L}^{\ell}L^j\bar{N}^k
     +\Hc{r}{\nabla_{\bar{L}}L}{N}+\Hc{r}{L}{\nabla_{L}N},\\
     \label{E:SwitchDerivates2}
     \bar{N}\Hc{r}{L}{L} &=
     \sum_{j,k,\ell=1}^2\frac{\partial^3 r}{\partial\bar{z}_{\ell}\partial z_j\partial\bar{z}_k}
     \bar{N}^{\ell}L^j\bar{L}^k
     +\Hc{r}{\nabla_{\bar{N}}L}{L}+\Hc{r}{L}{\nabla_{N}L}.     
   \end{align}
   Since $\Omega$ is weakly pseudoconvex at $p_0$, one has $\Hc{r}{L}{L}(p_0) = \Hc{r}{L}{N}(p_0) = 0$, i.e., $\Hc{r}{L}{\,\cdot\,}(p_0) \colon\Vhol{}{U} \to \R$ is identically zero. Moreover, since $0 \equiv L\bar{L}r = \Hc{r}{L}{L} + (\nabla_L\bar{L})r$, it follows that $((\nabla_L\bar{L})r)(p_0) = 0$, i.e., $(\nabla_L\bar{L})_{p_0} = cL_{p_0}$ for some constant $c \in \C$. Thus, the two rightmost terms in both \eqref{E:SwitchDerivates1} and \eqref{E:SwitchDerivates2} vanish at $p_0$, which proves that $\bar{L}\Hc{r}{L}{N}(p_0) = \bar{N}\Hc{r}{L}{L}(p_0)$.
   Since all tangential derivatives of $\Hc{r}{L}{L}$ vanish at $p_0$, it follows that $\nu\Hcomplex_r(L,L)(p_0) = 2\bar{N}\Hcomplex_r(L,L)(p_0) = 0$, which completes the proof. 
\end{proof}

This weaker result for boundary points of type greater than $4$ leads to a simplified proof of the Diederich--Forn{\ae}ss index and the Steinness index being $1$ for smoothly bounded, pseudoconvex domains which admit a smooth defining function that is plurisubharmonic on the boundary of the domain.

\begin{corollary}\label{C:DFforhighertype}
  Let $\Omega\Subset\mathbb{C}^2$ be a smoothly bounded, pseudoconvex domain. Suppose that 
  \begin{itemize}
    \item[(i)] $c_p\neq 4$ for all $p\in b\Omega$, and
    \item[(ii)]  $\Omega$ admits a smooth defining function which is plurisubharmonic on $b\Omega$.
  \end{itemize}
  Then for every $\eta\in(0,1)$ there exist a constant $K>0$ and an open neighborhood $U$ of $b\Omega$ such that
  $-(-r-Kr^2)^\eta$ is plurisubharmonic on $\Omega\cap U$.
  
  Similarly, for every $\mu>1$ there exist a constant $K>0$ and an open neighborhood $U$
  of $b\Omega$ such that $-(-r-Kr^2)^\mu$ is plurisubharmonic on $\Omega^c\cap U$.
\end{corollary}
We note that (ii) itself leads to the Diederich--Forn{\ae}ss index being 1, see \cite{ForHer07}. However,  the additional condition (i) simplifies the construction in \cite[Section 3]{ForHer07} considerably.

\begin{proof}
  Let $r$ be a smooth defining function of $\Omega$ which is plurisubharmonic on $b\Omega$, and assume
  that $L \coloneqq L_r$ and $N \coloneqq N_r$ are defined on some open neighborhood $U$ of $b\Omega$. 
  After possibly shrinking $U$, we may use \eqref{E:Taylor} for $f=\Hc{r}{L}{L}$ and $q\in U$ with $c_{\pi(q)}\geq 6$. It then follows from Lemma \ref{L:type6result} that
  \begin{align*}
   \Hc{r}{L}{L}(q)=\mathcal{O}(r^2)(q).
  \end{align*}
  Similarly, one obtains $ \Hc{r}{L}{N}(q)=\mathcal{O}(r)(q)$ and $ \Hc{r}{N}{N}(q)=\mathcal{O}(1)(q)$. It then follows 
  \begin{align*}
    \Hc{r}{\xi}{\xi}(q)=\left(\mathcal{O}(r^2)|\xi|^2+\mathcal{O}(|\langle\partial r,\xi\rangle|^2)\right)(q)\;\;\sjump\forall\; \xi\in\mathbb{C}^n\;\;\forall\;q\in U \text{ with } c_{\pi(q)}\ge 6.
  \end{align*}
  The arguments following (3.7) in \cite{ForHer07} then prove the claim.
  \end{proof}

\section{On a special class of pseudoconvex domains} \label{S:sesquiconvex}
In this section, we derive a sufficient condition for the existence of local defining functions, which are plurisubharmonic on the boundary, in terms of real coordinates. While this condition, in contrast to the criterion given in \Cref{L:pshonboundaryiff}, % and \Cref{L:pshnearboundaryiff},
 is not an equivalent characterization, it has the advantage of being independent of the choice of defining function, and thus is more easily checkable.
%In this section, we consider the condition of the existence of defining functions, which are plurisubharmonic on the boundary of the given domain, in real coordinates. This leads to a geometric boundary condition which is equivalent to any defining function, whose gradient is of constant
%length on the boundary, satisfying \eqref{E:pshonboundary}. However, this geometric condition is not implied by the existence of a smooth defining function which is plurisubharmonic on the boundary.

Let $\Omega\subset\mathbb{C}^2$ be a smoothly bounded domain, and let $r \colon U \to \R$ be a smooth local 
defining function for $\Omega$ near some point $p_0 \in b\Omega$. After possibly shrinking $U$, let $L$ be a nonvanishing holomorphic tangential vector field on $U$, and let $N=N_r$ be defined as in \eqref{E:definitionN}. Write
\begin{align*}
  L=\textstyle\frac{1}{2}(X+iY),\;\;N=\frac{1}{2}(\nu + iT)
\end{align*}
with $X,Y,T,\nu \in \Vr{}{U}$. The matrix associated with the real Hessian form $\Qrsymbol{r} \colon \Vr{}{U} \times \Vr{}{U} \to \R$ relative to the  basis $(X,Y,T,\nu)$ will be denoted by $\Qrsymbolcal{r}$, i.e.,
\begin{align*}
\Qrsymbolcal{r}\coloneqq
   \begin{pmatrix}
  \Hr{r}{X}{X} & \Hr{r}{X}{Y} &  \Hr{r}{X}{T} & \Hr{r}{X}{\nu}\\
  \Hr{r}{Y}{X} & \Hr{r}{Y}{Y} &  \Hr{r}{Y}{T} & \Hr{r}{Y}{\nu}\\
  \Hr{r}{T}{X} & \Hr{r}{T}{Y} &  \Hr{r}{T}{T} & \Hr{r}{T}{\nu}\\
   \Hr{r}{\nu}{X} & \Hr{r}{\nu}{Y} &  \Hr{r}{\nu}{T} & \Hr{r}{\nu}{\nu}
 \end{pmatrix}
 .
\end{align*}
We readily recognize that various convexity-like boundary conditions for $\Omega$ near $p_0$ may be expressed through conditions on entries of the leading principal $3\times 3$ submatrix of $\Qrsymbolcal{r}$ for $p\in b\Omega$ near $p_0$.
\begin{itemize}
  \item[(i)] $\Omega$ is convex near $p_0$ if the leading principal $3\times3$ submatrix of $\Qrsymbolcal{r}(p)$ is positive semi-definite for all $p\in b\Omega$ near $p_0$.
  \item[(ii)] $\Omega$ is $\mathbb{C}$-convex near $p_0$ if the leading principal $2\times 2$ submatrix of $\Qrsymbolcal{r}(p)$ is positive semi-definite for all $p\in b\Omega$ near $p_0$.
    \item[(iii)] $\Omega$ is pseudoconvex near $p_0$ if the trace of the leading principal $2\times 2$ submatrix of $\Qrsymbolcal{r}(p)$ is non-negative for all $p\in b\Omega$ near $p_0$.
\end{itemize}

In order to see how to express plurisubharmonicity on the boundary of a smooth local defining function in real coordinates, we need to formulate condition \eqref{E:pshonboundary} in real coordinates. Thus, we compute 
\begin{equation} \label{E:HrcLNreal} \begin{split}
  4\Hc{r}{L}{N}
  &= 4L(\bar{N}r)-4\left(\nabla_{L}\bar{N}\right)r\\
  &= (X+iY)(\nu-iT)r-\left(\nabla_{X+iY}(\nu-iT)\right)r\\
  &= X\nu r-(\nabla_{X}\nu)r+YT r-(\nabla_YT)r\\
  &\quad+i\bigl(Y\nu r-(\nabla_Y \nu)r-XTr+(\nabla_XT)r \bigr)\\
  &= \Hr{r}{X}{\nu}+\Hr{r}{Y}{T} + i\bigl(\Hr{r}{Y}{\nu} - \Hr{r}{X}{T}\bigr).
\end{split} \end{equation}
\Cref{L:pshonboundaryiff} may now be reformulated in terms of entries of $\Qrsymbolcal{r}$ as follows.

\begin{lemma}\label{L:pshinrealcoordinates}
  Let $\Omega\subset\mathbb{C}^2$ be a smoothly bounded, pseudoconvex domain, $p_0\in b\Omega$. Then 
  $\Omega$ admits a smooth local defining function which is plurisubharmonic on $b\Omega$ near $p_0$ 
  if and only if there exists a smooth local defining function $\rho \colon U \to \R$ for $\Omega$ 
  on some open neighborhood $U$ of $p_0$ such that
    \begin{align}\label{E:pshinrealcoordinates}
      \bigl(\Hr{\rho}{X}{\nu}+\Hr{\rho}{Y}{T} \bigr)^2+\bigl(\Hr{\rho}{Y}{\nu} - \Hr{\rho}{X}{T} \bigr)^2=\mathcal{O}\left(\Hc{\rho}{L}{L}\right) 
    \end{align}
    on $b\Omega\cap U$.
\end{lemma}

  In the proof of \Cref{T:type4pshonboundary} it is shown that, given any smooth local defining function 
  $r \colon U \to \R$ for $\Omega$, then $\rho \coloneqq re^h$ satisfies \eqref{E:pshinrealcoordinates} 
  if $h \in \mathcal{C}^\infty(U,\R)$ solves the equation 
  $$
  Lh = -\frac{2}{\abs{dr}}\Hc{r}{L}{N} + \mathcal{O}(\sqrt{\lambda}) \quad\text{on } b\Omega \cap U, 
  $$
  with $\lambda \coloneqq \Hc{r}{L}{L}_{|_{b\Omega \cap U}}$. 
  This can be reformulated in real coordinates as follows,
  \begin{align}
      Xh&=\frac{-1}{|dr|}\bigl(\Hr{r}{X}{\nu}+\Hr{r}{Y}{T}\bigr) +\mathcal{O}(\sqrt{\lambda}) \quad\text{on } b\Omega \cap U,\label{E:pshcondonhx}\\
      Yh&=\frac{1}{|dr|}\bigl(\Hr{r}{X}{T} - \Hr{r}{Y}{\nu}\bigr) +\mathcal{O}(\sqrt{\lambda}) \quad\text{on } b\Omega \cap U.\label{E:pshcondonhy}
    \end{align}

\begin{example}\label{E:examplecont}
  Let us revisit Example \ref{E:example}. There, in $\C^2$ with coordinates $z=x+iy$ and $w=u+iv$, 
  we consider $\Omega=\{(z,w)\in U: r(z,w)<0\}$ with
  \begin{align*}
    U=\{(z,w)\in\mathbb{C}^2:|x|<\pi/2\}, \quad r(z,w)=u-\textstyle\frac{1}{2}(x-v)^2-\ln(\cos(x)).
  \end{align*}
  We already computed that, with $L=r_w\frac{\partial}{\partial z}-r_z\frac{\partial}{\partial w}$,
 % \begin{align*}
 %   \GRAD r(z,w)=(v-x+\tan x,0,1,x-v),
 % \end{align*}
 % and 
  \begin{align*}
  \Hc{r}{L}{L}(z,w)=\textstyle\frac{1}{16}(x-v)^2\sec^2(x).
  \end{align*}
  In particular, note that the function $P(z,w) \coloneqq x-v$ is of class $\mathcal{O}(\sqrt{\lambda})$ on 
  $b\Omega$. Considering real vector fields on $\C^2$ as maps to $\R^4$, we can then compute further that
  \begin{align*}
  X &= \frac{1}{\abs{dr}}(r_u,-r_v,-r_x,r_y) 
    = \frac{1}{\abs{dr}}\left(1,0,-\tan(x),0\right) + \mathcal{O}(\sqrt{\lambda}), \\
  Y &= \frac{1}{\abs{dr}}(-r_v,-r_u,r_y,r_x) 
    = \frac{1}{\abs{dr}}\left(0,-1,0,\tan(x)\right) + \mathcal{O}(\sqrt{\lambda}), \\
  T &= \frac{1}{\abs{dr}}(r_y,-r_x,r_v,-r_u) 
    = \frac{1}{\abs{dr}}\left(0,-\tan(x),0,-1\right) + \mathcal{O}(\sqrt{\lambda}), \\
  \nu &= \frac{1}{\abs{dr}}(r_x,r_y,r_u,r_v) 
    = \frac{1}{\abs{dr}}\left(\tan(x),0,1,0\right) + \mathcal{O}(\sqrt{\lambda}), 
  \end{align*}
  where, in slight deviation from previous notation, the terms $\mathcal{O}(\sqrt{\lambda})$ denote vector fields 
  with coefficients that are of class $\mathcal{O}(\sqrt{\lambda})$ on $b\Omega$.
  From this, it follows readily that on $b\Omega$
  \begin{align*}
  \Hr{r}{X}{\nu} &= \frac{1}{\abs{dr}^2}\tan^3(x) + \mathcal{O}(\sqrt{\lambda}), &  \Hr{r}{Y}{T} &= \frac{1}{\abs{dr}^2}\tan(x) + \mathcal{O}(\sqrt{\lambda}), \\
  \Hr{r}{X}{T} &= -\frac{1}{\abs{dr}^2} + \mathcal{O}(\sqrt{\lambda}), &  \Hr{r}{Y}{\nu} &= \frac{1}{\abs{dr}^2}\tan^2(x) + \mathcal{O}(\sqrt{\lambda}).
  \end{align*}
  Since $|dr|^2 = 1+\tan^2(x) + \mathcal{O}(\lambda)$ on $b\Omega$,
  %  +\mathcal{O}(\lambda)(z,w) \quad\text{on } b\Omega
  %\begin{align*}
  %  \frac{1}{|dr|^2}(z,w)=\frac{1}{1+\tan^2(x)}
  %  +\mathcal{O}(\lambda)(z,w) \quad\text{on } b\Omega,
  %\end{align*}
  %and
  %\begin{align*}
  %\mathcal{Q}_r^\R (z,w)=
  %   \begin{pmatrix}
  %    \tan^2(x) & 0 &  0 & 1\\
  %     0 & 0 &  0 & 0\\
  %     0 & 0 &  0 & 0\\
  %     1 & 0 &  0& -1
  % \end{pmatrix}
  %\end{align*}
  it follows further that \eqref{E:pshcondonhx} and \eqref{E:pshcondonhy} are given by
  \begin{align*}
  h_x -\tan(x)h_u &= -\tan(x) + \mathcal{O}(\sqrt{\lambda}) \quad\text{on } b\Omega, \ \\
  -h_y + \tan(x)h_v &= -1 + \mathcal{O}(\sqrt{\lambda}) \quad\text{on } b\Omega.
  \end{align*}
  It is easy to see that, e.g., $h(z,w) = y+\ln(\cos(x))$ and $h(z,w) =y+u$ both satisfy these last two equations, so that $re^h$ satisfies \eqref{E:pshinrealcoordinates}.  Hence, if $\rho=re^h$, then for every $V \Subset U$ there exists a smooth defining function for $\Omega$ on $U$ that is plurisubharmonic on $b\Omega \cap V$, see \Cref{R:pshonboundaryiff}. In view of \Cref{T:pshnearboundary}, this means that $\Omega$ admits plurisubharmonic smooth local defining functions near each boundary point. 
\end{example}

\begin{definition}
Let $\Omega \subset \C^2$ be a smoothly bounded domain. We say that $\Omega$ is sesquiconvex at 
$p_0 \in b\Omega$ if $\Omega$ is pseudoconvex at $p_0$ and if there exists a smooth local defining 
function $\rho \colon U \to \R$ for $\Omega$ near $p_0$ such that
\begin{align} \label{E:distancepshC}
   \Hr{\rho}{Y}{T}^2+\Hr{\rho}{X}{T}^2=\mathcal{O}\left(\Hc{\rho}{L}{L}\right) \quad \text{on } b\Omega \cap U.
\end{align}
\end{definition}
\begin{remark}
(1) Let $\rho \colon U \to \R$ be a smooth local defining function for $\Omega$ near $p_0$, 
and let $h \colon U \to \R$ be smooth. Then for any two tangential vector fields $V,W$ near $p_0$ 
one has $\Hr{\rho e^h}{V}{W} = - (\nabla_VW)(\rho e^h) = -e^h(\nabla_VW)\rho = e^h\Hr{\rho}{V}{W}$ on $b\Omega \cap U$. 
In particular, this shows that condition \eqref{E:distancepshC} is independent of the choice of a local defining function $\rho$.

(2) Since $\Hc{\rho}{L}{L} = \Hr{\rho}{X}{X} + \Hr{\rho}{Y}{Y}$, one easily sees that every domain $\Omega \subset \C^2$ that is convex at $p_0 \in b\Omega$ is sesquiconvex at $p_0$. Moreover, it is clear from the definition that if $\Omega$ is strictly pseudoconvex at $p_0 \in b\Omega$, then $\Omega$ is sesquiconvex at $p_0$. On the other hand, the domain $\Omega$ considered in \Cref{E:example} and \Cref{E:examplecont} is %$\C$-convex at $0 \in b\Omega$ but 
not sesquiconvex at $0$.
%, while one easily checks that the domain $\Omega' \coloneqq \{(z,w) \in \C^2 : u + x^2 - \frac{1}{2}y^2 < 0\}$, $z = x+iy, w= u+iv$, is sesquiconvex at $0 \in b\Omega'$ but not $\C$-convex at $0$.
\end{remark}

In the following, we show that sesquiconvexity at a boundary point $p_0 \in b\Omega$ implies the existence of local defining functions which are plurisubharmonic on a one-sided neighborhood $\bar{\Omega} \cap U$ of $p_0$, i.e., $\Hcsymbolcal{\rho}(q) \ge 0$ for every $q \in \bar{\Omega} \cap U$.
%In the following, we will say that a smooth function $\rho \colon U \to \R$ is plurisubharmonic on $\bar{\Omega} \cap U$ if $\Hcsymbolcal{\rho} \ge 0$ on $\bar{\Omega} \cap U$.

\begin{proposition}\label{P:sesquiconvexpsh}
If $\Omega \subset \C^2$ is sesquiconvex at $p_0 \in b\Omega$, then $\Omega$ admits a smooth local defining function $\rho:U\longrightarrow\mathbb{R}$ near $p_0$ which is plurisubharmonic
on $\bar{\Omega}\cap U$.
\end{proposition}

\begin{proof}
Let $r \colon U \to \R$ be a smooth local defining function for $\Omega$ near $p_0$, and assume that $dr \neq 0$ on $U$. We will show that $\rho \coloneqq r/|dr|$ satisfies
  \begin{align}
    \left|\Hc{\rho}{L}{N}\right|&=\mathcal{O}(\sqrt{\lambda})\;\;\text{on }b\Omega\cap U,\text{ and}\label{E:pshnearboundaryiff1b}\\
  \nu\Hc{\rho}{L}{L}&\leq\mathcal{O}(\sqrt{\lambda})\;\;\text{on } b\Omega\cap U.\label{E:pshnearboundaryiff2b}
  \end{align}
The claim then follows from a brief analysis of the proofs of \Cref{L:basicestimate1} and \Cref{L:pshnearboundaryiff}.

Let $L$ be a nonvanishing holomorphic tangential vector field on $U$, and let $X,Y \in \Vr{}{U}$ such that $L = \frac{1}{2}(X+iY)$. A straightforward computation shows that
\begin{align}\label{E:class1condition}
  0=L(|d\rho|^2)=\Hr{\rho}{X}{\nu}+i\Hr{\rho}{Y}{\nu}  \quad\text{on } b\Omega \cap U.
\end{align}
Since $\Omega$ is sesquiconvex at $p_0$, it thus follows from \eqref{E:distancepshC} and the computations in \eqref{E:HrcLNreal} that \eqref{E:pshnearboundaryiff1b} is true. The equation in \eqref{E:class1condition} may be expressed in complex notation, with $N= 2\rho_{\bar{z}^1}\frac{\partial}{\partial z^1}+2 \rho_{\bar{z}^2}\frac{\partial}{\partial z^2}$,
as
$$
0 = L(|\partial\rho|^2) =  \Qc{\rho}{L}{N}+\Hc{\rho}{L}{N} \;\;\text{on } b\Omega \cap U.
$$
Since \eqref{E:pshnearboundaryiff1b} holds, it then follows that
\begin{align}\label{E:pshnearboundaryiff1c}
   \left|\Qc{\rho}{L}{N}\right|=\mathcal{O}(\sqrt{\lambda})\;\;\text{on }b\Omega\cap U.
\end{align}
Moreover, we compute on $b\Omega\cap U$
\begin{align*}
  0&=\bar{L}L\left(|\partial\rho|^2\right)
=\bar{L}\left( \Qc{\rho}{L}{N}\right)+\bar{L}\left(\Hc{\rho}{L}{N}\right)\\
&=\sum_{j,k,\ell=1}^2\frac{\partial^{3}\rho}{\partial \bar{z}^\ell \partial z^j\partial z^k}\bar{L}^\ell L^j N^k
+\Qc{\rho}{\nabla_{\bar{L}}L}{N}+\Qc{\rho}{L}{\nabla_{\bar{L}}N}\\
&+\sum_{j,k,\ell=1}^2\frac{\partial^{3}\rho}{\partial \bar{z}^\ell \partial z^j \partial \bar{z}^k}\bar{L}^\ell L^j \bar{N}^k
+\Hc{\rho}{\nabla_{\bar{L}}L}{N}+\Hc{\rho}{L}{\nabla_{L}N}.
\end{align*}
In view of \eqref{E:SwitchDerivates2}, it follows from \eqref{E:pshnearboundaryiff1b} that on $b\Omega \cap U$
\begin{align*}
2\re\left(\sum_{j,k,\ell=1}^2\frac{\partial^{3}\rho}{\partial \bar{z}^\ell \partial z^j\partial z^k}\bar{L}^\ell L^j N^k\right)
=2\re(N)\Hc{\rho}{L}{L}+\mathcal{O}(\sqrt{\lambda}).
\end{align*}
Moreover, it follows from \eqref{E:pshnearboundaryiff1b} that 
$$
\left|\Hc{\rho}{L}{\nabla_{L}N}\right|=\mathcal{O}(\sqrt{\lambda}) \quad\text{on } b\Omega\cap U.
$$
Furthermore, since $\Hc{\rho}{L}{L}=-(\nabla_{\bar{L}}{L})\rho$ on $b\Omega\cap U$, it follows that the normal component of $\nabla_{\bar{L}}{L}$ is 
$\mathcal{O}(\Hc{\rho}{L}{L})$ on $b\Omega\cap U$. This, together with \eqref{E:pshnearboundaryiff1b} and \eqref{E:pshnearboundaryiff1c}, implies that
\begin{align*}
%   \left|\Qc{\rho}{\nabla_{\bar{L}}L}{N}\right|=\mathcal{O}\left(\sqrt{\Hc{\rho}{L}{L}} \right)=\left|\Hc{\rho}{\nabla_{\bar{L}}L}{N}\right|\quad\text{on } b\Omega \cap U.
   \left|\Qc{\rho}{\nabla_{\bar{L}}L}{N}\right| &= \mathcal{O}(\sqrt{\lambda}) \;\;\text{on } b\Omega \cap U, \text{ and} \\
   \left|\Hc{\rho}{\nabla_{\bar{L}}L}{N}\right| &=\mathcal{O}(\sqrt{\lambda}) \;\;\text{on } b\Omega \cap U.
\end{align*}
Therefore we obtain
\begin{align*}
  2\re(N)\Hc{\rho}{L}{L}=-\Qc{\rho}{L}{\nabla_{\bar{L}}N}+\mathcal{O}(\sqrt{\lambda})\;\;\text{on } b\Omega \cap U.
\end{align*}
Since $$\nabla_{\bar{L}}N=2\sum_{j,k=1}^2\bar{L}^k\rho_{\bar{z}^j\bar{z}^k}\frac{\partial}{\partial z^j},$$
it follows easily that $\Qc{\rho}{L}{\nabla_{\bar{L}}N}=|\Qc{\rho}{L}{.\,}|^2$, where $\Qc{\rho}{L}{.\,} \in \Vhol{}{U}$ is defined via the condition $\langle \Qc{\rho}{L}{.\,}, V\rangle \coloneqq \Qc{\rho}{L}{V}$ for all $V \in \Vhol{}{U}$. Hence
\begin{align*}
   \nu\Hc{\rho}{L}{L}=2\re(N)\Hc{\rho}{L}{L}=-|\Qc{\rho}{L}{.\,}|^2+\mathcal{O}(\sqrt{\lambda})\;\;\text{on } b\Omega \cap U,
\end{align*}
i.e., \eqref{E:pshnearboundaryiff2b} holds and the claim follows.
\end{proof}

A global analog of \Cref{P:sesquiconvexpsh} easily follows for sesquiconvex, bounded domains.
\begin{corollary} \label{T:sesquiconvexpshglobal}
  If $\Omega\Subset\mathbb{C}^2$ is sesquiconvex, then $\Omega$ admits a smooth global defining function $\rho:U\longrightarrow\mathbb{R}$
which is plurisubharmonic on $\bar{\Omega}\cap U$.
\end{corollary}

\begin{remark}
Let $\Omega \subset \C^2$ be a smoothly bounded domain, let $N=N_\rho$ for some smooth local defining function $\rho \colon U \to \R$ for $\Omega$, see \eqref{E:definitionN}, and let $L$ be a nonvanishing holomorphic tangential vector field on $U$. A straightforward computation shows that then
\begin{align*}
 \nabla_{\bar{L}}N = \frac{\overline{\Qc{\rho}{L}{L}}}{2|L|^2}  L + \frac{\overline{\Qc{\rho}{L}{N}}}{2|N|^2}  N \;\;\text{on } b\Omega \cap U,
\end{align*}
and hence, in particular,
\begin{align*}
\Qc{\rho}{L}{\nabla_{\bar{L}}N} = \frac{\abs{\Qc{\rho}{L}{L}}^2}{2|L|^2} + \frac{\abs{\Qc{\rho}{L}{N}}^2}{2|N|^2} \;\;\text{on } b\Omega \cap U.
\end{align*}
Moreover, it is easy to see that if $\Omega$ is $\C$-convex at $p_0 \in b\Omega \cap U$, then $|\Qc{\rho}{L}{L}| = \mathcal{O}(\lambda)$ on $b\Omega$ near $p_0$. In view of \eqref{E:pshnearboundaryiff1c}, it thus follows from the arguments in the proof of \Cref{P:sesquiconvexpsh}, that $\Omega$ admits a plurisubharmonic smooth local defining function near $p_0$ whenever $\Omega$ is both sesquiconvex and $\C$-convex at $p_0$. Similarly, if $\Omega \Subset \C^2$ is sesquiconvex and $\C$-convex, then $\Omega$ admits a plurisubharmonic smooth defining function.
\end{remark}

%\begin{remark}
%Since $\Hc{\rho}{L}{L} = \Hr{\rho}{X}{X} + \Hr{\rho}{Y}{Y}$, one easily sees that every domain $\Omega \subset \C^2$ that is convex at $p_0 \in b\Omega$ is sesquiconvex at $p_0$. Moreover, it is clear %from the definition that if $\Omega$ is strictly pseudoconvex at $p_0 \in b\Omega$, then $\Omega$ is sesquiconvex at $p_0$. On the other hand, the domain $\Omega$ considered in \Cref{E:example} and %
%\Cref{E:examplecont} is $\C$-convex at $0 \in b\Omega$ but not sesquiconvex at $0$, while one easily checks that the domain $\Omega' \coloneqq \{(z,w) \in \C^2 : u + x^2 - \frac{1}{2}y^2 < 0\}$, $z = x+iy, %w= u+iv$, is sesquiconvex at $0 \in b\Omega'$ but not $\C$-convex at $0$.
%\end{remark}

%
%

\bibliographystyle{acm}
\bibliography{References}

\end{document}